\documentclass[a4paper]{article}
\usepackage[textwidth=14.5cm]{geometry}
\usepackage{graphicx}
\usepackage{amsmath,amssymb,amsthm}
\usepackage[bookmarks=false]{hyperref}
\usepackage{authblk}
\newcommand{\til}{\widetilde}
\newcommand{\appu}{z}
\newcommand{\appp}{r}
\newcommand{\eps}{\varepsilon}
\newcommand{\R}{\mathbf R}
\newcommand{\pk}[1]{\mathrm{P}_{#1}}
\newcommand{\dzero}{\delta}
\newcommand{\ch}{s_\delta}
\newcommand{\szero}{s_1}
\newcommand{\snorm}[1]{|#1|_{s}}
\newcommand{\pro}[1]{\widehat #1_h}
\theoremstyle{plain}
\newtheorem{theorem}{Theorem}[section]
\newtheorem{lemma}[theorem]{Lemma}
\theoremstyle{definition}
\newtheorem{remark}[theorem]{Remark}
\theoremstyle{definition}
\newtheorem{hypo}{Hypothesis}
\newcommand{\mone}{$\blacktriangle$}
\newcommand{\mtwo}{\Large$\bullet$\normalsize}
\newcommand{\mthree}{$\blacksquare$}
\newcommand{\mfour}{$\vartriangle$}
\newcommand{\mfive}{\Large$\circ$\normalsize}
\newcommand{\msix}{$\square$}
\begin{document}

\title{A pressure-stabilized projection Lagrange--Galerkin scheme for the transient Oseen problem}
\author{Shinya Uchiumi} 
\affil{
	Department of Mathematics, Gakushuin University, Tokyo 171-8588, Japan\\ 
	\texttt{shinya.uchiumi@gakushuin.ac.jp}
}

\date{November 7, 2021}

\maketitle

\begin{abstract}
We propose and analyze a pressure-stabilized projection Lagrange--Galerkin scheme for the transient Oseen problem. 
The proposed scheme inherits the following advantages from the projection Lagrange--Galerkin scheme.
The first advantage is computational efficiency. 
The scheme decouples the computation of each component of the velocity and pressure.
The other advantage is essential unconditional stability.
Here we also use the equal-order approximation for the velocity and pressure, and add a symmetric pressure stabilization term. 
This enriched pressure space enables us to obtain accurate solutions for small viscosity.
First, we show 
an error estimate for the velocity for small viscosity.
Then we show 
convergence results for the pressure.
Numerical examples of a test problem show higher accuracy of the proposed scheme for small viscosity.

\end{abstract}
%
%
\section{Introduction}
We consider a finite element scheme for the transient Oseen problem, known as a linearization of the Navier--Stokes (NS) problem, with small viscosity.
We need special cares to obtain
accurate numerical solutions even in this linear problem.

We focus on the Lagrange--Galerkin (LG) method, which combines the method of the characteristics and Galerkin method.
The LG method
is a robust numerical technique for solving convection-dominated flow problems. 
It was first developed and analyzed in \cite{DouglasRussell1982,Pironneau1982}, and 
the analysis for the NS problem was improved in \cite{Suli}.
The LG method is also applied to, e.g., natural convection problems \cite{BenitezBermudez2011} and viscoelastic models \cite{LukacovaNonlin,LukacovaLin}.
One advantage of this method is the explicit treatment of the convection term so that the resulting matrix is symmetric.
Moreover,
the scheme is essentially unconditionally stable for the Oseen problem \cite{NotsuTabataOseen2015}.
It means that stability conditions, such as $\Delta t \leq c h^\alpha$, are not needed, where $\Delta t$ is the time increment, $h$ is the mesh size, and $c$ and $\alpha$ are positive constants.
We note that this stability is not influenced by small viscosity for the Oseen problem \cite{NotsuTabataOseen2015}.

One of the main ingredients of this paper is the combination of LG and  fractional-step projection methods. 
See \cite{Guermond2006} for the overview of the projection method.
The main advantage is the computational efficiency that decouples the velocity and pressure.
Achdou and Guermond \cite{AchdouGuermond2000} have proposed a combined scheme of the incremental pressure correction projection and LG method using the inf-sup stable elements for the NS problem. They have derived error estimates for the velocity and pressure when the viscosity constant is 1.
They used the solution of a system of ordinary differential equations (ODEs) as the trajectory map,
which should be approximated in practical computation.
Misawa \cite{Misawa2016} has considered an Euler approximated scheme of \cite{AchdouGuermond2000} and error estimates of the same order have been derived for the velocity and pressure.
However, in \cite{GuermondMinev2003}, instability of the scheme \cite{AchdouGuermond2000} has been observed 
for relatively large time increment and Reynolds number.
There, they have implemented the scheme \cite{AchdouGuermond2000} with some approximation and have computed a flow behind a backward-facing step.
Guermond and Minev \cite{GuermondMinev2003} have developed an LG/projection scheme for the NS problem, which is stable under the same condition, and derived an error estimate for the velocity.
The estimate for the pressure, however, is not known to the best of the author's knowledge.

Here we also focus on the dependence on the inverse of the viscosity.
The above estimates are of the forms $c(\Delta t^\alpha + h^\beta)$.
We note that the constant $c$ contains not only a Sobolev norm of the exact solution, $\|(u,p)\|_{X_1}$ 
but also a norm multiplied by the inverse of the viscosity, e.g., $\nu^{-1} \|(u,p)\|_{X_2}$.
The effect of $\nu^{-1}$ in the latter term appears even when the exact solution does not show sharp boundary layers.
See a recent survey \cite{GJN2021kine}.

One choice of eliminating the effect of the inverse of the viscosity is to enhance  the divergence-free condition (mass conservation) by
the grad-div stabilization \cite{Franca1988}. 
Error analyses independent of
the viscosity were performed for the Stokes problem \cite{OlshanskiiReusken}, 
the transient Oseen problem \cite{Burman2017,deFrutos2016},
and the transient NS problem \cite{DeFrutos2019JSC}
by relying on this term.
However, a drawback is that the grad-div operator creates coupled matrices for the velocity \cite{Linke2013}.

Recently, without the grad-div stabilization, 
dependence on the inverse of the viscosity can be eliminated only by using equal-order pairs of finite elements 
with pressure stabilization, for the transient problems.
Chen and Feng \cite{ChenFeng2017} have analyzed a semi-discrete scheme using the equal-order element with symmetric pressure stabilization for the transient NS problem to derive uniform error estimates with respect to the Reynolds number.
De Frutos et al. \cite{DeFrutos2019IMA} have analyzed a standard Galerkin scheme for the transient NS problem using the equal-order finite elements with local projection stabilization. 
Such estimates also hold for an LG scheme for the transient Oseen equations with the equal-order elements with the stabilization of Brezzi--Pitk\"aranta or its generalization to the higher order element \cite{Uchiumi2019}.

In this paper, we propose a projection/LG scheme using the equal-order element with a pressure stabilization for the transient Oseen problem.
The advantages of computational efficiency and essentially unconditional stability are inherited from the projection/LG scheme.
The projection/LG part is based on Guermond and Minev for the inf-sup stable elements \cite{GuermondMinev2003}.
The pressure stabilized fractional-step projection part is mainly adopted from Burman et al.~\cite{Burman2017}.
See also comments in Subsection \ref{subsec:commentOnScheme} below.
Firstly, 
we derive error estimates for the velocity 
in $L^2$-norm of order $\Delta t + h^k$
independent of the inverse of the viscosity,
where $k$ is the degree of piecewise polynomials.
Then,
we show an error estimate of order $\Delta t + h^k$ for the pressure, which may depend on the viscosity.
It is worth noting that even viscosity-dependent pressure estimates in the Oseen framework
have not been obtained for the scheme \cite{GuermondMinev2003} to the best of the author's knowledge.
The technical difficulty is, as in \cite{deFrutos2016,DeFrutos2019JSC,Garcia-archilla2020}, 
the estimate of the time difference of the velocity.

We mention related works.
Burman et al.~\cite{Burman2017} developed and analyzed a projection scheme for the Oseen problem.
They used the equal-order elements with the continuous interior penalty method, and with terms including the grad-div and pressure stabilization.
Robust error estimates with respect to the Reynolds numbers are derived for the velocity and a time-average of the pressure.
The order is $\Delta t + h^{k+1/2}$ in $L^2$-norm for the velocity.
The optimal estimate of order $\Delta t+h^{k+1}$ independent of the viscosity is not known so far \cite{GJN2021kine}.
De Frutos et al.~\cite{DeFrutos2019JSC,DeFrutos2019JSCCorri} proposed and analyzed a projection scheme for the NS problem.
They used the inf-sup stable standard Galerkin method with the grad-div stabilization.
They derived viscosity-independent error estimates for the velocity.
It seems difficult to get the viscosity-independent estimate for the pressure with optimal order.
Garc\'ia-Archilla et al.~\cite{Garcia-archilla2020} analyzed the implicit Galerkin scheme with equal-order element and pressure stabilization 
for the NS problem
and derived error estimates independent of the inverse of the viscosity.
Badia and Codina \cite{Badia2007} have developed and analyzed a non-incremental projection scheme  for the NS problem using the equal-order element with a local projection type stabilization.
De Frutos et al.~\cite{deFrutos2018projNS,deFrutos2018projStokes} 
have analyzed 
a projection scheme with non inf-sup stable elements for the Stokes and the NS problem.
The stabilization term is the same as the discretized Laplacian in the pressure Poisson equation,
thus, extra stabilization is not necessary.
In these works \cite{Badia2007,deFrutos2018projNS,deFrutos2018projStokes},
the condition $\Delta t \sim h^2$ is needed,
and the error constant depends on the inverse of the viscosity.

The remainder of the paper is organized as follows.
In the next section we state the Oseen problem
and present a pressure-stabilized projection LG scheme with preparing notation.
In Section \ref{sec:errorEstiSmallVis}
we show error estimates for the velocity with small viscosity
and their proof.
In Section \ref{sec:errorEstiP}
we show an error estimate for the pressure
and its proof.
In Section \ref{sec:numericalResults} we give some numerical results,
where the Taylor--Hood pair and equal-order ones are compared.
In Section \ref{sec:conclusions} we give conclusions.
In the appendix section we prove a lemma used in the LG methods.

\section{Problem setting and a present scheme}
\subsection{Continuous problem}
We prepare notation used throughout this paper, and state the Oseen problem.

Let $\Omega$ be a polygonal or polyhedral domain of $\R^d ~ (d=2,3)$.
We use the Sobolev spaces $W^{m,p}(\Omega)$ equipped with the norm $\|\cdot\|_{m,p}$ and the seminorm $|\cdot|_{m,p}$ for $p\in [1,\infty]$ and a non-negative integer $m$.  
We denote $W^{0,p}(\Omega)$ by $L^p(\Omega)$.  
The space $W^{1,p}_0(\Omega)$ consists of functions in $W^{1,p}(\Omega)$ whose traces vanish on the boundary of $\Omega$.
When $p=2$, we denote $W^{m,2}(\Omega)$ by $H^m(\Omega)$ and drop the subscript $2$ in the corresponding norm and seminorm.
For the vector-valued function $w\in W^{1,\infty}(\Omega)^d$ we define the seminorm $|w|_{1,\infty}$ by
\begin{equation*}
\biggl\lVert\biggl[\sum_{i,j=1}^d \left( \frac{\partial w_i}{\partial x_j} \right)^2 \biggr]^{1/2}\biggr\rVert_{0,\infty}.
\end{equation*}
The pair of parentheses $(\cdot , \cdot)$ shows the $L^2(\Omega)^i$-inner product for $i=1, d$ or $d\times d$.
The space $L^2_0(\Omega)$ consists of functions $q \in L^2(\Omega)$ satisfying $(q,1)=0$. 
The dual space of $H^1_0(\Omega)$ is denoted by $H^{-1}(\Omega)$ with the norm $\| \cdot \|_{-1}$. 
We also use the notation $|\cdot|_{m,K}$ and $(\cdot,\cdot)_K$ for the seminorm and the inner product on a set $K$, respectively.

Let $T>0$ be a time.
For a Sobolev space $X(\Omega)^i$, $i=1, d$,  
we use the abbreviations 
$H^m(X)=H^m(0,T;X(\Omega)^i)$ and $C(X)=C([0,T];X(\Omega)^i)$.
We define the function space $Z^m$ by
\begin{equation*}
\begin{split}
	Z^m &:= \{ v\in H^j(0,T; H^{m-j}(\Omega)^d); ~ j=0,\dots,m, \| v \|_{Z^m}<\infty \},\\
	\| v \|_{Z^m} &:= 
	\biggl( \sum_{j=0}^m \|v\|_{H^j(0,T;H^{m-j}(\Omega)^d)}^2 \biggr) ^{1/2}.
\end{split}
\end{equation*}
We also use the notation $H^m(t_1, t_2; X)$ and $Z^m(t_1,t_2)$ for spaces on a time interval $(t_1, t_2)$.

We consider the Oseen problem: find
$(u,p):\Omega \times (0,T) \to \R^d \times \R$ such that
\begin{equation}\label{Os} 
\begin{split}
\frac{\partial u}{\partial t}+(w \cdot \nabla) u 
- \nu \Delta u + \nabla p = f & \quad \text{in} \quad \Omega \times (0,T),\\
\nabla \cdot u = 0 & \quad \text{in} \quad \Omega \times (0,T),\\
u = 0 & \quad \text{on} \quad \partial \Omega \times (0,T),\\
u(\cdot,0) = u^0 & \quad \text{in} \quad \Omega,
\end{split}
\end{equation}
where 
$\partial \Omega$ represents the boundary of $\Omega$,
the constant $\nu (>0)$ represents a viscosity, and
$w, f:\Omega \times (0,T) \to \R^d$ and
$u^0:\Omega \to \R^d$ are given functions.
We assume $w=0$ on $\partial \Omega$.

We define the bilinear forms $a$ on $H^1_0(\Omega)^d\times H^1_0(\Omega)^d$ 
by 
\begin{equation*}
	a(u,v) := \nu(\nabla u,\nabla v).
\end{equation*}
Then, we can write the weak form of \eqref{Os} as follows:  
find $(u,p) : (0,T)\to H^1_0(\Omega)^d \times L^2_0(\Omega)$ such that for $t\in (0,T)$, 
\begin{subequations}\label{eq:Osweak}
\begin{alignat}{2}
	\left( \Bigl( \frac{\partial u}{\partial t} + (w \cdot \nabla) u \Bigr) (t),v\right) + a(u(t),v)
	+ (\nabla p(t), v) 
	&= (f(t),v),  
	&& ~ \forall v\in H^1_0(\Omega)^d,  \label{eq:Osweaka} \\
	(\nabla \cdot u(t),q ) &= 0, 
	\quad && \forall q\in L^2_0(\Omega),
\end{alignat}
\end{subequations}
with $u(0)=u^0$.

\subsection{Temporal discretization}
Let 
$\Delta t>0$ be a time increment, 
$N_T := \lfloor T/\Delta t \rfloor$ the number of time steps, 
$t^n := n\Delta t$, and 
$\psi^n := \psi(\cdot,t^n)$ for a function $\psi$ defined in $\Omega \times (0,T)$.

Let $w$ be smooth. 
The characteristic curve $X(t; x,s)$ 
is defined by the solution of the system of the ODEs, 
\begin{equation} \label{eq:xode}
\begin{split}
	\frac{dX}{dt}(t; x,s)&=w(X(t;x,s),t), \quad t<s,\\
	X(s;x,s)&=x.
\end{split}
\end{equation}
Then, we can write the material derivative term $\frac{\partial u}{\partial t}+ (w\cdot \nabla) u$ as follows:
\begin{equation*}
\left(\frac{\partial u}{\partial t} + (w \cdot \nabla) u \right)(X(t),t)
=\frac{d}{dt} u(X(t),t).
\end{equation*}
For $w^*:\Omega \to \R^d$ we define the mapping $X_1(w^*):\Omega \to \R^d$ by
\begin{equation}\label{eq:x1def}
(X_1(w^*))(x) := x - w^*(x)\Delta t.
\end{equation}
\begin{remark}
	The image of $x$ by $X_1(w(\cdot,t))$ is the approximate value 
	of $X(t-\Delta t;x,t)$ obtained by solving \eqref{eq:xode} by the backward Euler method.
\end{remark}
Then, it holds that
\begin{equation*}
\frac{\partial u^n}{\partial t}+(w^n\cdot \nabla) u^n = \frac{u^n-u^{n-1}\circ X_1(w^{n-1})}{\Delta t} + O(\Delta t),
\end{equation*}
where the symbol $\circ$ stands for the composition of functions, 
e.g., $(g\circ f)(x) := g(f(x))$.

\subsection{Spatial discretization}
Let $\{\mathcal T_h\}_{h \downarrow 0}$ be a regular family of triangulations of $\overline \Omega$ \cite{Ciarlet}, $h_K:=\operatorname{diam}(K)$ for an element $K\in \mathcal T_h$, and $h:= \max_{K\in \mathcal T_h} h_K$. 
For a positive integer $m$, the finite element space of order $m$ is defined by
\begin{equation*} 
	W_h^{(m)} := \{ \psi_h \in C(\overline \Omega); ~ \psi_{h|K} \in \pk m(K), ~ \forall K \in \mathcal T_h \},
\end{equation*}
where $\pk m(K)$ is the set of polynomials on $K$ whose degrees are equal to or less than $m$.
For a pair of positive integers $(k, \ell)$, we define 
$\pk{k}/\pk{\ell}$-finite element space
by 
\[
	V_h \times Q_h := [W_h^{(k)} \cap H_0^1(\Omega)]^d \times [W_h^{(\ell)} \cap L_0^2(\Omega)].
\]
The space $Y_h:=V_h + \nabla Q_h$
is also used in the projection method \cite{Guermond1996}.
We denote by $i_h^T$ the $L^2$-projector from $Y_h$ to $V_h$, 
which is the dual operator of $i_h$, the injection from $V_h$ to $Y_h$.

For the equal order $\pk{k}/\pk{k}$-element we use a symmetric positive semidefinite bilinear form $\szero: Q_h \times Q_h \to \R$ for stabilization,
which is specified in Hypothesis \ref{hypo:stab1} below.
A typical example 
is
\begin{equation}
\begin{split}
	\szero(p_h, q_h) := \sum_{K\in \mathcal T_h} h_K^{2k} \sum_{|\alpha|=k} (D^\alpha p_h, D^\alpha q_h)_K,
\end{split}
\label{eq:stabBPB}
\end{equation}
which is the stabilization by Brezzi and Pitk\"aranta \cite{BrezziPitkaranta} for the $\pk1/\pk1$-element
and its extension to higher order elements \cite{Burman2008}.
We note that $\szero$ does not include the viscosity or a stabilization parameter.
Practically a non-negative parameter $\dzero$ is included in the stabilization term as follows:
\[
	\ch (p_h, q_h) := \dzero \szero(p_h, q_h).
\]

\subsection{Present scheme}
We are now in position to define our pressure-stabilized projection LG scheme called Scheme($k, \ell, \dzero$).
Let $\mathcal T_h$, $\Delta t$, integers $k, \ell \geq 1$ and a real number $\dzero \geq 0$
be given.

{\bf Scheme($k, \ell, \dzero$)}:
Let $V_h \times Q_h$ be
the $\pk{k}/\pk{\ell}$-finite element space on $\mathcal T_h$.
Let $(u_h^0, p_h^0) \in Y_h \times Q_h$ be given.
Find $(\til u_h^n, u_h^n, p_h^n) \in V_h \times Y_h \times Q_h$,
$n=1,...,N_T$ such that
for $n=0,1,...,N_T-1$
\begin{subequations} \label{eq:mainscheme}
\begin{align}
	\biggl( \frac{\til u_h^{n+1} - (i_h^T u_h^{n}) \circ X_1(w^{n})}{\Delta t}, v_h \biggr) 
	+ a(\til u_h^{n+1} ,v_h) + (\nabla p_h^{n}, v_h) &= (f^{n+1}, v_h), 
	 &&\forall v_h \in V_h, 
	\label{subeq:momentum}\\
	\frac{u_h^{n+1} - \til u_h^{n+1}}{\Delta t} + \nabla(p_h^{n+1} - p_h^{n}) &= 0, &&
	\label{subeq:yh}\\
	(u_h^{n+1}, \nabla q_h) - \ch(p_h^{n+1}, q_h) &= 0, 
	&&\forall q_h \in Q_h. \label{subeq:stab}
\end{align}
\end{subequations}

We later see in Lemma \ref{lemm:bijectiveMix} that,
for each $w^n \in W^{1,\infty}_0(\Omega)^d$ and
under the condition $\Delta t |w^n|_{1,\infty} < 1$,
the inclusion $(X_1(w^n))(\Omega) \subset \Omega$ holds.
Thus the composite function $(i_h^T u_h^n) \circ X_1(w^n)$ is well-defined on $\Omega$.

While we use $i_h^T u_h^n$ in the practical scheme,
the variable $u_h^{n+1}$ is eliminated.
For $n \geq 1$,
we obtain the following practical scheme
using the expression of $u_h^{n}$ in \eqref{subeq:yh},
testing \eqref{subeq:yh} with $\nabla q_h$ and using \eqref{subeq:stab}.

Stage 1: find $i_h^T u_h^n \in V_h$ such that
\begin{equation}
	(i_h^T u_h^n, v_h) = (\til u_h^n - \Delta t \nabla(p_h^n - p_h^{n-1}), v_h), \quad \forall v_h \in V_h.
	\label{eq:schemePracProj}
\end{equation}
Stage 2: find $\til u_h^{n+1} \in V_h$ such that
\begin{alignat}{2}
	\frac{1}{\Delta t} (\til u_h^{n+1}, v_h) + a(\til u_h^{n+1} ,v_h) 
	= \frac{1}{\Delta t} \left( (i_h^T u_h^n) \circ X_1(w^{n}) , v_h \right) 
	- \left( \nabla p_h^{n}, v_h \right)  
	+ (f^{n+1}, v_h), \notag \\
	\quad \forall v_h \in V_h. \label{eq:schemePraca}
\end{alignat}
Stage 3: find $p_h^{n+1} \in Q_h$ such that
\begin{equation}
	(\nabla p_h^{n+1}, \nabla q_h) 
	+ \frac{1}{\Delta t} \ch(p_h^{n+1}, q_h) = 
	(\nabla p_h^{n}, \nabla q_h) 
	+ \frac{1}{\Delta t} (\til u_h^{n+1} ,\nabla q_h),
	\quad \forall q_h \in Q_h.
	\label{eq:schemePracb}
\end{equation}

\subsection{Comments on the scheme}
\label{subsec:commentOnScheme}
This scheme has an advantage in computational cost.
We can decouple the left-hand side of \eqref{eq:schemePracProj}, and \eqref{eq:schemePraca} into each velocity component, respectively, as follows:
\[
	(\til U_{h1}, v_{h1}) + (\til U_{h2}, v_{h2}), 
\]
\[
\Bigl( \frac{1}{\Delta t}(\til u_{h1}^{n+1}, v_{h1})
+ a(\til u_{h1}^{n+1} ,v_{h1}) \Bigr)
+ \Bigl( \frac{1}{\Delta t} (\til u_{h2}^{n+1}, v_{h2})
+ a(\til u_{h2}^{n+1} ,v_{h2}) \Bigr),
\]
where  
$(\til U_{h1}, \til U_{h2}) = i_h^T u_h^n$,
$( \til u_{h1}^{n+1}, \til u_{h2}^{n+1} ) = \til u_h^{n+1}$
and $( v_{h1}, v_{h2}) = v_h$ if $d=2$.
Each part corresponds to the matrix in the discretized Poisson equation with the mass term,
which is easy to handle by linear solvers such as the conjugate gradient method \cite{templates}.
The matrix in \eqref{eq:schemePracb} is the one in the discretized Poisson equation subject to the Neumann boundary condition with the symmetric stabilization term.
In addition to the explicit treatment of the convection term,
we do not need time restriction such as $\Delta t \leq ch^\alpha$ for the error estimates in this paper.

We adopt the framework of Guermond and Minev \cite{GuermondMinev2003}
for the combination of the projection and the LG methods.
Achdou and Guermond \cite{AchdouGuermond2000} also developed 
the combined scheme for the NS equations.
To show the corresponding formulation in \cite{AchdouGuermond2000},
we replace \eqref{subeq:momentum} by
\[
	\biggl( \frac{\til u_h^{n+1} - u_h^{n}}{\Delta t}, v_h \biggr) + \biggl(  \frac{\til u_h^{n} - \til u_h^{n} \circ X_1(w^{n})}{\Delta t}, v_h \biggr)
	+ a(\til u_h^{n+1} ,v_h) + (\nabla p_h^{n}, v_h) = (f^{n+1}, v_h).
\]
Then \eqref{eq:schemePraca} is replaced by the following form without  $i_h^T u_h^n$:
\begin{equation*}
\begin{split}
	\frac{1}{\Delta t} (\til u_h^{n+1}, v_h) + a(\til u_h^{n+1} ,v_h)  
	= \frac{1}{\Delta t}(\til u_h^{n} \circ X_1(w^{n}), v_h)
	- (\nabla(2p_h^{n} - p_h^{n-1}), v_h) 
	+ (f^{n+1}, v_h).
\end{split}
\end{equation*}
However, 
it is difficult to derive the viscosity robust error estimate in Section \ref{sec:errorEstiSmallVis}.

Following Guermond and Minev \cite{GuermondMinev2003}, in \eqref{subeq:momentum}, we do not use the original  $u_h^n \in Y_h$ but the $L^2$-projection $i_h^T u_h^n \in V_h$.
Our error estimates in Sections \ref{sec:errorEstiSmallVis} and \ref{sec:errorEstiP} are also valid if we replace $i_h^T u_h^n$ by $u_h^n$.
However, from the implementation viewpoint, 
we need to integrate the term $([\til u_h^n - \nabla (p_h^{n} - p_h^{n-1}) ] \circ X_1(w^{n}), v_h)$ in view of \eqref{subeq:yh}.

In the scheme of Burman et al.~\cite{Burman2017}, 
the pressure stabilization is defined a functional on $Y_h$.
Here we simply define $\ch$ as the bilinear form on $Q_h$.

Finally, we mention an implementation issue.
It is difficult to compute the term $((i_h^T u_h^n) \circ X_1(w^n), v_h)$ because the integrand is not polynomial on each element.
It is known that rough numerical quadrature leads to instability.
A remedy is to introduce a locally linearized velocity $w_h^n$, i.e., the $\mathrm P_1$-Lagrange interpolation of $w^n$ \cite{TSTeng}.
Then the term 
\begin{equation}
	((i_h^T u_h^n) \circ X_1(w_h^n), v_h)
	\label{eq:intLLV}
\end{equation}
can be exactly computable.
The error estimates of this paper can be done with the Lagrange interpolation error of $O(h^2)$ \cite{Tabata2018,Uchiumi2019}.

\section{Error estimates for the velocity with small viscosity}
\label{sec:errorEstiSmallVis}

We use $c$ to represent a
generic positive constant that is independent of $\nu$, $\Delta t$, $h$ and $\dzero$
but depends on Sobolev norms of $w$, $u$ and $p$, and $T$,
and may take a different value at each occurrence.

\subsection{Hypotheses and the main theorem for the velocity}
\label{subsec:orderhypoandmaintheo}

\begin{hypo}\label{hypo:uestiRegSmallvis}
The velocity $w$ and
the exact solution $(u, p)$ of the Oseen problem \eqref{Os} satisfy
\[
	w \in C(W^{1,\infty}_0) \cap H^1(L^\infty), \quad	
	u \in Z^2 \cap H^1(H^k) \cap C(H^{k+1}), \quad
	p \in H^1(H^1) \cap C(H^{k+1}).
\]
\end{hypo}

\begin{hypo} \label{hypo:dt}
The time increment $\Delta t$ satisfies $0<\Delta t \leq \Delta t_0$, where
\[
	\Delta t_0 := \frac{1}{4 |w|_{C(W^{1,\infty})}}.
\]
\end{hypo}

\begin{hypo} \label{hypo:stab1}
The bilinear form $\szero$ satisfies the following conditions.
\begin{enumerate}
\item \label{item:stab1-sym}
$\szero: Q_h \times Q_h \to \R$
is a symmetric and positive semidefinite bilinear form.
\item For all $q_h \in Q_h$, 
\[
	\szero(q_h, q_h) \leq c \|q_h\|_0^2.
\]
\item \label{item:stab1-stabilty}
There exists an operator $\Pi_h: L^2_0(\Omega) \to Q_h$ such that
\begin{align}
	\|q - \Pi_h q\|_{m} &\leq c h^{s+1-m} \|q\|_{s+1}, ~ q \in L^2_0(\Omega) \cap H^{s+1}(\Omega), ~ 0 \leq s \leq k, 
	~ m=0,1.
	\label{eq:s0projapprox} \\
	\szero(\Pi_h q, \Pi_h q)^{1/2} &\leq ch^{s} \| q \|_{s}, ~ q \in L^2_0(\Omega) \cap H^s(\Omega), ~ 0 \leq s \leq k. 
	\label{eq:s0projstability}
\end{align}
\item \label{item:stab1-divergence}
There is an operator $\mathcal I_h^k: V^{\mathrm{div}} \to V_h$ such that 
for all $v \in V^{\mathrm{div}}$ and $q_h \in Q_h$,
\begin{align}
	|(\nabla \cdot (v - \mathcal I_h^k v), q_h)| 
	&\leq c \biggl( \sum_{K \in \mathcal T_h} h_K^{-2} \| v - \mathcal I_h^k v \|_{0, K}^2 + \| v - \mathcal I_h^k v \|_1^2 \biggr)^{1/2} \szero(q_h, q_h)^{1/2}, 
	\label{eq:hypoVQmix}\\
	\| v - \mathcal I_h^{k} v\|_m 
	&\leq c h^{s+1-m} |v|_{s+1}, ~ \forall v \in H^{s+1}(\Omega)^d, 
	~  1 \leq s \leq k, 
	~ m=0,1.
	\label{eq:hypoVQvinterpo}
\end{align}
Here, $V^{\mathrm{div}} = \{ v \in H^1_0(\Omega)^d; ~ \nabla \cdot v = 0\}$.
\end{enumerate}
\end{hypo}

In view Hypothesis \ref{hypo:stab1}-(\ref{item:stab1-sym}),
we define the seminorm by
\[
	\snorm{q_h} = \szero(q_h, q_h)^{1/2}, \quad \forall q_h \in Q_h.
\]
Then, the Schwarz inequality holds:
\[
	\szero (q_h, r_h) \leq \snorm{q_h} \snorm{r_h}, \quad \forall q_h, r_h \in Q_h.
\]
From \eqref{eq:hypoVQmix} and \eqref{eq:hypoVQvinterpo} we easily get
\begin{equation} 
	|(\nabla \cdot (v - \mathcal I_h^k v), q_h)|  
	\leq ch^k |v|_{k+1} \snorm{q_h} \quad \forall v \in V^{\mathrm{div}} \cap H^{k+1}(\Omega)^d, q_h \in Q_h.
	\label{eq:hypo1practical}
\end{equation}

The term $\szero$ in \eqref{eq:stabBPB} and $\pk{k}/\pk{k}$-element ($k \geq 1$) space $V_h \times Q_h$ satisfy Hypothesis \ref{hypo:stab1} with $\Pi_h$ being the Cl\'ement interpolation \cite{Clement},
and $\mathcal I_h^k$ being a modified Stokes projection \cite{deFrutos2016} when $k\geq 2$, or Lagrange interpolation when $k=1$.
See \cite{Garcia-archilla2020,Uchiumi2019}.
We note that the constant does not depend on the viscosity.

\begin{remark}
Hypothesis \ref{hypo:stab1} is mainly adopted from \cite{Garcia-archilla2020} although \eqref{eq:s0projapprox} is slightly stronger.
As pointed out there, these assumptions are quite similar to those in \cite{BurmanFernandez2008}.
We refer to \cite{Garcia-archilla2020} for other stabilization satisfying Hypothesis \ref{hypo:stab1}.
\end{remark}

\begin{hypo}\label{hypo:initSmallVis}
The initial value $(u_h^0, p_h^0)$ is chosen so that 
there exists a positive constant $c$ independent of $h$ such that
\[
	\| u_h^0 - u^0 \|_0 \leq c h^k, \quad
	\| \nabla (p_h^0 - p^0) \|_0 \leq c.
\]
\end{hypo}

For a set of functions $\psi=\{ \psi^n \}_{n=0}^{N_T}$
we use two norms 
$\|\cdot\|_{\ell^\infty(L^2)}$ and $\|\cdot\|_{\ell^2(L^2)}$
and a seminorm $|\cdot|_{\ell^2(s)}$
defined by
\begin{equation}
\begin{split}
\|\psi\|_{\ell^\infty(L^2)} & := 
\max \left\{ \|\psi^n\|_{  0  };n=0,\dots ,N_T \right\}, \\
\|\psi\|_{\ell^2(L^2)} & := 
\biggl(
\Delta t \sum_{n=1}^{N_T} \|\psi^n\|_{0}^2 
\biggr)^{1/2}, \quad
|\psi|_{\ell^2(s)} := \biggl(
\Delta t \sum_{n=1}^{N_T} | \psi^n |_{s}^2 
\biggr)^{1/2}.
\end{split}
\label{eq:discNorms}
\end{equation}

\begin{theorem}\label{thm:veloSmallVis}
Let $(u_h, p_h):=\{(u_h^n, p_h^n)\}_{n=0}^{N_T}$ be the solution of 
Scheme$(k, k, \dzero)$ 
with $k \geq 1$ and $\dzero >0$.
Assume Hypothesis \ref{hypo:uestiRegSmallvis}, 
\ref{hypo:dt}, \ref{hypo:stab1}, \ref{hypo:initSmallVis}.
Then the following estimate holds:
\begin{equation}
\begin{split}
	& \| u_h - u\|_{\ell^\infty(L^2)}, ~
	\| \til u_h - u\|_{\ell^\infty(L^2)}, ~
	\nu^{1/2} \| \nabla(\til u_h - u ) \|_{\ell^2(L^2)}, ~
	\dzero^{1/2} |p_h - \Pi_h p|_{\ell^2(s)}
	\\
	\leq & c( 1 + \nu^{1/2} + \dzero^{1/2} + \dzero^{-1/2})(\Delta t + h^k).
\end{split}
	\label{eq:velocitySmallVis}
\end{equation}
Here, $\Pi_h p$ is the interpolation in Hypothesis \ref{hypo:stab1}.
\end{theorem}

\subsection{Preliminaries for the velocity estimates} 
\label{subsec:prelV}
We use the techniques developed in the finite element projection method \cite{Guermond1996,Guermond1998}, 
LG method \cite{RuiTabata2002,Tabata2018},
combined method \cite{AchdouGuermond2000,GuermondMinev2003},
pressure-stabilized method \cite{BurmanFernandez2008,FrancaStenberg1991},
and projection with pressure-stabilized method \cite{Burman2017}.

For a function $F$ defined on $[0,T]$,
or sequence of functions $F = \{F^n\}_{n=0}^{N_T}$,
\begin{align}
	d_t F &:= F - F(\cdot - \Delta t), \text{ or} \notag \\
	d_t F^n &:= F^n - F^{n-1}.
	\label{eq:dtdef}
\end{align}
Let $F \in H^1(X)$ or $H^2(X)$ for a Banach space $X(\Omega)^i$, $i=1,d$.
The following inequalities are frequently used:
\begin{align}
	\| d_t F^n \|_X &\leq \Delta t^{1/2} \Bigl\| \frac{\partial F}{\partial t} \Bigr\|_{L^2(t^{n-1}, t^n; X)}, 
	\label{eq:dthalfest}
	\\
	\| d_t F^n - d_t F^{n-1}\|_X & \leq c \Delta t^{3/2} \Bigl\| \frac{\partial^2 F}{\partial t^2} \Bigr\|_{L^2(t^{n-2}, t^n; X)}.
	\label{eq:dtsechalfest}
\end{align}

First we recall a discrete version of the Gronwall inequality.
\begin{lemma}[discrete Gronwall inequality] 
\label{lem:discreteGronwall}
Let $\gamma_1$ be a non-negative number, $\Delta t$ be a positive number,
and $\{x^n\}_{n\geq n_0}, \{y^n\}_{n\geq n_0+1}$ and $\{b^n\}_{n\geq n_0+1}$ be non-negative sequences.
Suppose
\begin{equation*}
	\frac{x^n-x^{n-1}}{\Delta t} + y^n 
	\leq 
	\gamma_1 x^{n-1} + b^n, ~ \forall n \geq n_0+1.
\end{equation*}
Then, it holds that 
\begin{equation*}
	x^n + \Delta t \sum_{i=n_0+1}^n y^i 
	\leq 
	\exp [ \gamma_1 (n-n_0)\Delta t ]
	\left( x^{n_0} + \Delta t \sum_{i=n_0+1}^n b^i \right), 
	~ \forall n \geq n_0+1.
\end{equation*}
\end{lemma}
Lemma \ref{lem:discreteGronwall} is shown by using the inequalities
\begin{equation*}
\begin{split}
	x^n + y^n \Delta t 
	&\leq (1+\gamma_1 \Delta t) x^{n-1} + b^n \Delta t  
	\leq \exp(\gamma_1\Delta t)(x^{n-1}+b^n\Delta t).
\end{split}
\end{equation*}
Instead of the well-known summation form of the discrete Gronwall inequality, e.g., in \cite{HeywoodRannacher1990}, we use this form 
because the condition on $\Delta t$ does not include $\gamma_1$, making the proof simpler.

We prepare fundamental properties of the mapping $X_1(w^*)$ for $w^*\in W^{1,\infty}_0(\Omega)^d$.
We refer to \cite{RuiTabata2002,Tabata2018} for the proofs.

\begin{lemma}
\label{lemm:bijectiveMix}
Let $w^*\in W^{1,\infty}_0(\Omega)^d$ and $X_1(w^*)$ be the mapping defined in \eqref{eq:x1def}.
\begin{enumerate}
\item \label{item:bijective}
Under the condition $\Delta t |w^*|_{1,\infty} < 1$,
it holds that $(X_1(w^*))(\Omega) \subset \Omega$ and $X_1(w^*): \Omega \to \Omega$ is bijective.
\item \label{item:jacobiest}
Under the condition $\Delta t |w^*|_{1,\infty} \leq 1/4$,
the estimate 
\begin{equation*}
	\frac{1}{2} \leq 
	J
	\leq \frac{3}{2}
\end{equation*}
holds,
where $J$ is the Jacobian of $X_1(w^*)$.
\item Under the condition $\Delta t|w^*|_{1,\infty} \leq 1/4$,
there exists a positive constant $c$
independent of $\Delta t$
such that for $v \in L^2(\Omega)^d$ 
\begin{equation*}
	\| v \circ X_1(w^*) \|_{0}^{2}
	\leq (1+ c |w^*|_{1,\infty} \Delta t) \|v\|_0^{2}.
\end{equation*}
\end{enumerate}
\end{lemma}

Lemma \ref{lemm:errorEqArrange} is fundamental to establishing the stability in the error equations for the projection methods.
We give a proof for completeness
although it
is natural extension of the classical argument \cite{Guermond1996,Guermond1998}, and
is derived in  a similar way to Lemma 4.1 in \cite{Burman2017}.

\begin{lemma} \label{lemm:errorEqArrange}
Let $\{\til U_h^n\}_{n=n_0}^{n_1} \subset V_h$, 
$\{U_h^n\}_{n=n_0}^{n_1} \subset Y_h$,
$\{G^n\}_{n=n_0}^{n_1} \subset L^2(\Omega)^d$, 
and $\{P_{h}^n\}_{n=n_0}^{n_1}$,  $\{\Psi_{h}^n \}_{n=n_0}^{n_1} \subset Q_h$, satisfy
for $n=n_0,n_0+1,...,n_1-1$,
\begin{subequations}
\begin{align}
	\biggl( \frac{\til U_h^{n+1} - G^n}{\Delta t}, v_h \biggr) + a(\til U_h^{n+1}, v_h) + (\nabla \Psi_{h}^n, v_h) 
	&= \langle F^{n+1}, v_h \rangle, 
	~ \forall v_h \in V_h, \label{align:errorEqGena}\\
	\frac{U_h^{n+1} - \til U_h^{n+1}}{\Delta t}+ \nabla (P_{h}^{n+1} - \Psi_{h}^n) &=0,  \label{align:errorEqGenb}\\
	(U_h^{n+1}, \nabla q_h) - \ch(P_{h}^{n+1}, q_h) &= \langle S^{n+1}, q_h \rangle,  
	~ \forall q_h \in Q_h, \label{align:errorEqGenc}
\end{align}
\end{subequations}
with 
$F^{n+1}$ and 
$S^{n+1}$ being linear functionals on $V_h$ and $Q_h$, respectively.
Then, it holds that for $n=n_0, n_0+1,...,n_1-1$
\begin{equation} \label{eq:errorEqArrange}
\begin{split}
	&\frac{1}{2\Delta t} ( \| U_h^{n+1} \|_0^2 -  \| G^n \|_0^2 + \| \til U_h^{n+1} - G^n \|_0^2) \\
	&+ \nu \| \nabla \til U_h^{n+1}  \|_0^2 
	+ \frac{\Delta t}{2} ( \|\nabla  P_{h}^{n+1} \|_0^2 - \| \nabla \Psi_{h}^n \|_0^2 )
	+ \dzero \snorm{P_{h}^{n+1}}^2 \\
	= & \langle F^{n+1}, \til U_h^{n+1} \rangle - \langle S^{n+1}, P_{h}^{n+1} \rangle.
\end{split}
\end{equation}
\end{lemma}

\begin{proof}
The equation \eqref{align:errorEqGena} with $v_h = \til U_h^{n+1}$, 
and the identity $(a-b)a = \frac{1}{2} a^2 - \frac{1}{2} b^2 + \frac{1}{2} (a-b)^2$ yields
\begin{equation} 
\begin{split}
	&\frac{1}{2\Delta t} ( \| \til U_h^{n+1} \|_0^2 -  \| G^n \|_0^2 + \| \til U_h^{n+1} - G^n \|_0^2) \\
	&+ \nu \| \nabla \til U_h^{n+1}  \|_0^2  
	+ (\nabla \Psi_{h}^{n}, \til U_h^{n+1}) 
	= \langle F^{n+1}, \til U_h^{n+1} \rangle.
\end{split}
	\label{eq:errorEqArrangeNoNum1}
\end{equation}
Testing \eqref{align:errorEqGenb} with $\Delta t \nabla \Psi_{h}^n$ and using the identity $(a-b)b = \frac{1}{2} a^2 - \frac{1}{2} b^2 - \frac{1}{2} (a-b)^2$, and again using \eqref{align:errorEqGenb}, 
we have
\begin{equation}
	(U_h^{n+1} - \til U_h^{n+1}, \nabla \Psi_{h}^n) +  \frac{\Delta t}{2} ( \|\nabla  P_{h}^{n+1} \|_0^2 - \| \nabla \Psi_{h}^n \|_0^2 ) 
	- \frac{1}{2\Delta t} \| U_h^{n+1} - \til U_h^{n+1} \|_0^2 =0.
	\label{eq:errorEqArrangeNoNum2}
\end{equation}
Testing \eqref{align:errorEqGenb} with $U_h^{n+1}$ yields
\begin{equation}
\begin{split}
	\frac{1}{2\Delta t} (\|U_h^{n+1} \|_0^2 - \| \til U_h^{n+1} \|_0^2 &+ \| U_h^{n+1} - \til U_h^{n+1} \|_0^2 ) \\
	&+ (\nabla P_{h}^{n+1}, U_h^{n+1}) - (\nabla \Psi_{h}^n, U_h^{n+1}) = 0.
\end{split}
	\label{eq:errorEqArrangeNoNum3}
\end{equation}
Finally, \eqref{align:errorEqGenc} with $q_h = P_{h}^{n+1}$ yields
\begin{equation}
	(U_h^{n+1}, \nabla P_{h}^{n+1}) - \dzero \snorm{P_{h}^{n+1}}^2 = \langle S^{n+1}, P_{h}^{n+1} \rangle.
	\label{eq:errorEqArrangeNoNum4}
\end{equation}
Adding \eqref{eq:errorEqArrangeNoNum1}--\eqref{eq:errorEqArrangeNoNum3} and subtracting \eqref{eq:errorEqArrangeNoNum4},
we have the conclusion \eqref{eq:errorEqArrange}.
\end{proof}

\begin{remark}
In Lemma 4.1 of \cite{Burman2017}, the corresponding estimate is not based on $U_h^n$ but on $i_h^T U_h^n$,
and the stabilization term is functional on $Y_h$.
\end{remark}

\subsection{Proof of Theorem \ref{thm:veloSmallVis}}
Let $(\appu_h(t), \appp_h(t)) \in V_h \times Q_h$ be the interpolation $(\mathcal I_h^k u(t), \Pi_h p(t))$ in Hypothesis \ref{hypo:stab1}.
We use the following notation:
\begin{subequations}
\begin{align}
e_h^n &= u_h^n - \appu_h^n, & \til e_h^n &= \til u_h^n - \appu_h^n, & \eta(t) &= u(t) - z_h(t), \\
\eps_h^n &= p_h^n - \appp_h^n, & \psi_h^n &= p_h^n - \appp_h^{n+1}, &
X_1^n &= X_1(w^n).
\end{align}
\label{eq:errornotation}
\end{subequations}

We begin with error equations in $e_h^n$, $\til e_h^n$ and $\eps_h^n$.
Connecting \eqref{subeq:momentum} and \eqref{eq:Osweaka} at $t = t^{n+1}$, subtracting
\begin{align*}
	\biggl( \frac{\appu_h^{n+1} - \appu_h^{n}\circ X_1^n}{\Delta t}, v_h \biggr)
	+ a(\appu_h^{n+1}, v_h) + (\nabla r_h^{n+1}, v_h), \\
	(\appu_h^{n+1}, \nabla q_h) - \ch(\appp_h^{n+1}, q_h)
\end{align*}
from both sides of \eqref{subeq:momentum} (equaling \eqref{eq:Osweaka}) and \eqref{subeq:stab}, respectively,
and noting $i_h^T z_h^n = z_h^n$,
we get the following error equation for $n=0,1,...,N_T-1$.
\begin{subequations}
\label{eq:erroreq1}
\begin{alignat}{2}
	\biggl( \frac{\til e_h^{n+1} - (i_h^T e_h^{n}) \circ X_1^n}{\Delta t}, v_h \biggr) + a(\til e_h^{n+1} ,v_h) + (\nabla \psi_h^{n}, v_h) 
	& = 
	\bigl\langle R_1^{n+1}, v_h \bigr\rangle, 
	~ &&\forall v_h \in V_h, 
	\label{subeq:erroreq1a}\\
	\frac{e_h^{n+1} - \til e_h^{n+1}}{\Delta t} + \nabla(\eps_h^{n+1} - \psi_h^{n}) &= 0, \label{subeq:erroreq1b}\\
	(e_h^{n+1}, \nabla q_h) - \ch(\eps_h^{n+1}, q_h)  &= 
	\langle S_1^{n+1}, q_h \rangle, 
	~ &&\forall q_h \in Q_h, \label{subeq:erroreq1c}
\end{alignat}
\end{subequations}
where 
\begin{align}
	\bigl \langle R_1^{n+1}, v_h \bigr\rangle &:= ( R_{11}^{n+1} + R_{12}^{n+1}, v_h )  
	+ a(\eta^{n+1}, v_h) + (\nabla (p^{n+1} - \appp_h^{n+1}), v_h),  
	\notag 
	\\
	R_{11}^{n+1} &:= \frac{\partial u^{n+1}}{\partial t} + (w^{n+1} \cdot \nabla) u^{n+1}  
	- \frac{u^{n+1} - u^{n} \circ X_1^{n}}{\Delta t}, \label{eq:R11}\\
	R_{12}^{n+1} &:= \frac{\eta^{n+1} - \eta^{n} \circ X_1^{n}}{\Delta t}, \label{eq:R12}
	\\
	\bigl\langle S_1^{n+1}, q_h \bigr\rangle &:= -(\appu_h^{n+1}, \nabla q_h) + \ch(\appp_h^{n+1}, q_h).
	\notag
\end{align}

Since the estimates of $R_{11}^n$ and $R_{12}^n$ are obtained by standard techniques used in the LG method (e.g. Lemmas 8, 10 in \cite{Uchiumi2019}), 
we omit their proofs.
\begin{lemma} \label{lemm:estR1}
Suppose that 
$w \in C(W^{1,\infty}_0) \cap  H^1(L^\infty)$
and
$\Delta t |w|_{C(W^{1,\infty})} \leq 1/4$.
Then, there exists a positive constant $c$ depending on the norm $\|w\|_{C(L^{\infty})}$ such that
	\begin{align*}
	&\|R_{11}^n\|_0  \leq c \sqrt{\Delta t}  \biggl(   \|u\|_{Z^2(t^{n-1},t^n)} +  \left\| \frac{\partial w}{\partial t} \right\|_{L^2(t^{n-1}, t^n; L^\infty)} \|\nabla u^n\|_0 \biggr), 
	\quad \forall u\in Z^2, \\
	&\left\| \frac{v^n - v^{n-1} \circ X_1^{n-1}}{\Delta t} \right\|_0 
	\leq \frac{c}{\sqrt{\Delta t }} \left\| v \right\|_{H^1(t^{n-1}, t^n ;L^2) \cap L^2(t^{n-1}, t^n; H^1)},
	\quad \forall v \in H^1(L^2) \cap L^2(H^1).
\end{align*}
\end{lemma}

\begin{proof}[Proof of Theorem \ref{thm:veloSmallVis}]
We apply Lemma \ref{lemm:errorEqArrange} to \eqref{eq:erroreq1} 
and obtain
\begin{align}
	& \frac{1}{2\Delta t} (\| e_h^{n+1} \|_0^2 -  \| (i_h^T e_h^{n}) \circ X_1^n \|_0^2 + \| \til e_h^{n+1} - (i_h^T e_h^{n}) \circ X_1^n \|_0^2) 
	+ \nu \| \nabla \til e_h^{n+1} \|_0^2  \notag \\
	& 
	+ \dzero \snorm{\eps_h^{n+1}}^2
	+ \frac{\Delta t}{2} (\| \nabla \eps_h^{n+1} \|_0^2 -  \| \nabla \psi_h^{n} \|_0^2) 
	= \langle R_1^{n+1}, \til e_h^{n+1} \rangle - \langle S_1^{n+1}, \eps_h^{n+1} \rangle.
	\label{eq:prth1erroreq}
\end{align}

For the estimate of $\| \til e_h^{n+1} \|_0$
we fix a $\gamma_0$ such that $\Delta t_0 \leq \frac{1}{8\gamma_0}$.
In the following we will get the bounds for $|\langle R_1^{n+1}, \til e_h^{n+1} \rangle|$ and $|\langle S_1^{n+1}, \eps_h^{n+1} \rangle|$ in the same way as in \cite{Uchiumi2019}.
From the Schwarz's inequality, 
\begin{equation}
	|(R^{n+1}_{1i}, \til e_h^{n+1})| \leq \frac{1}{\gamma_0} \|R_{1i}^{n+1}\|_0^2 + \frac{\gamma_0}{4} \|\til e_h^{n+1}\|_0^2, ~ i=1,2.
	\label{eq:R3Schwarz}
\end{equation}
Estimates for $\|R_{1i}^{n+1}\|_0^2$, $i=1,2$, are obtained by Lemma \ref{lemm:estR1} 
with $v=\eta$, 
and the following estimate is obtained by \eqref{eq:hypoVQvinterpo}:
\[
	\| \eta \|_{H^1(t^{n}, t^{n+1}; L^2) \cap L^2(t^{n}, t^{n+1}; H^{1})} 
	\leq ch^k \| u \|_{H^1(t^{n}, t^{n+1}; H^k) \cap L^2(t^{n}, t^{n+1}; H^{k+1})}.
\]
Bounds for the other terms in $\langle R_1^{n+1}, \til e_h^{n+1} \rangle$ are easily obtained from \eqref{eq:hypoVQvinterpo} and \eqref{eq:s0projapprox}:
\begin{align}
	|a(\eta^{n+1}, \til e_h^{n+1})| \leq \frac{\nu}{2} \|\nabla \eta^{n+1}\|_{0}^2 & + \frac{\nu}{2} \|\nabla \til e_h^{n+1}\|_{0}^2 
	\leq c \nu h^{2k} \| u^{n+1} \|_{k+1}^2 + \frac{\nu}{2} \|\nabla \til e_h^{n+1}\|_{0}^2, \notag \\
	|(\nabla(p^{n+1} - r_h^{n+1}), \til e_h^{n+1})| 
	& \leq \frac{1}{\gamma_0} \| \nabla (p^{n+1} - r_h^{n+1}) \|_0^2 + \frac{\gamma_0}{4} \|\til e_h^{n+1}\|_0^2 \notag \\
	& \leq  ch^{2k} \| p^{n+1} \|_{k+1}^2 + \frac{\gamma_0}{4} \|\til e_h^{n+1}\|_0^2.
	\label{eq:nablapEst}
\end{align}
Bounds for the terms in $\langle S_1^{n+1}, \eps_h^{n+1} \rangle$ are obtained 
from \eqref{eq:hypo1practical} and \eqref{eq:s0projstability} as follows:
\begin{align*}
	|(z_h^{n+1}, \nabla \eps_h^{n+1})| &= |(\nabla \cdot (z_h^{n+1} - u^{n+1}), \eps_h^{n+1})|
	\leq ch^{k} \| u^{n+1} \|_{k+1} \snorm{\eps_h^{n+1}}
	\notag \\
	& \leq \frac{c}{\dzero} h^{2k} \| u^{n+1} \|_{k+1}^2 
		+ \frac{\dzero}{4} \snorm{\eps_h^{n+1}}^2. 
	\\
	|\ch(r_h^{n+1}, \eps_h^{n+1})| 
	&\leq 
	\dzero \snorm{r_h^{n+1}} \snorm{\eps_h^{n+1}}
	\notag \\
	& \leq
	\dzero \snorm{r_h^{n+1}}^2 + \frac{\dzero}{4} \snorm{\eps_h^{n+1}}^2
	\leq c\dzero h^{2k} \|p^{n+1}\|_{k}^2 
	+ \frac{\dzero}{4} \snorm{\eps_h^{n+1}}^2.
\end{align*}
For $\|\til e_h^{n+1}\|_0^2$ in \eqref{eq:R3Schwarz} and \eqref{eq:nablapEst}
\begin{equation}
	\gamma_0 \| \til e_h^{n+1} \|_0^2 \leq  
	2\gamma_0 \| \til e_h^{n+1} - (i_h^T e_h^n) \circ X_1^n \|_0^2
	+ 2\gamma_0 \| (i_h^T e_h^n) \circ X_1^n \|_0^2,
	\label{eq:etilEst}
\end{equation}
and since $2\gamma_0 \leq \frac{1}{4 \Delta t_0} < \frac{1}{2 \Delta t}$, 
the first term is absorbed by the left hand side of \eqref{eq:prth1erroreq}.
From Lemma \ref{lemm:bijectiveMix}
and since $i_h^T$ is the $L^2$-projector,
\begin{equation}
	\| (i_h^T e_h^n) \circ X_1^n \|_0^2 
	\leq (1+c\Delta t) \| i_h^T e_h^n \|_0^2
	\leq (1+c\Delta t) \| e_h^n \|_0^2.
	 \label{eq:ehxEst}
\end{equation}
The estimate for $\| \nabla \psi_h^n \|_0^2$ is obtained by \eqref{eq:dthalfest} 
as follows:
\begin{equation}
\begin{split}
	\| \nabla \psi_h^n \|_0^2
	&= \| \nabla \eps_h^n - \nabla(r_h^{n+1} - r_h^n) \|_0^2  \\
	& \leq (1+\Delta t) \|\nabla \eps_h^n \|_0^2 + \Bigl( 1+\frac{1}{\Delta t} \Bigr) \|\nabla(r_h^{n+1} - r_h^n)\|_0^2, 
	\\
	& \leq (1+\Delta t) \|\nabla \eps_h^n \|_0^2 + c \|r_h\|_{H^1(t^n, t^{n+1}; H^1)}^2.
\end{split}
\label{eq:pdtesti}
\end{equation}
We note that $\|r_h\|_{H^1(t^{n},t^{n+1}; H^1)} \leq c \|p\|_{H^1(t^{n},t^{n+1}; H^1)}$ by \eqref{eq:s0projapprox}.

Combining these estimates, 
from \eqref{eq:prth1erroreq},
we now obtain
for $n=0,1,...,N_T-1$, 
\begin{equation}
	\frac{x^{n+1} - x^n}{\Delta t} + y^{n+1} \leq c x^n + c b^{n+1},
	\label{eq:gronwallprepare3}
\end{equation}
where
\begin{align}
	x^n = & \| e_h^n \|_0^2 + \Delta t^2 \| \nabla \eps_h^n \|_0^2, 
	\notag\\ 
	y^n = & 
	\frac{1}{2\Delta t} \| \til e_h^{n} - (i_h^T e_h^{n-1}) \circ X_1^{n-1} \|_0^2 + 
	\nu \| \nabla \til e_h^n \|_0^2 
	+ \dzero \snorm{\eps_h^n}^2, 
	\notag \\
	b^n = & \Delta t ( \| u \|_{Z^2(t^{n-1},t^n)}^2 + \| w \|_{H^1(t^{n-1}, t^n; L^\infty)}^2 + \| p \|_{H^1(t^{n-1}, t^{n}; H^1)}^2 ) 
	\notag \\
	& + \frac{c h^{2k}}{\Delta t} \| u \|_{H^1(t^{n-1}, t^n; H^k) \cap L^2(t^{n-1}, t^n; H^{k+1})}^2 
	+ c ( 1 + \nu + \dzero + \dzero^{-1} ) h^{2k},
	\notag
\end{align}
and $c$ is a constant depending on the Sobolev norms of $u$, $p$ and $w$, the constants in Hypothesis \ref{hypo:stab1} and in Lemmas \ref{lemm:bijectiveMix} and \ref{lemm:estR1}.
We apply Lemma \ref{lem:discreteGronwall} to \eqref{eq:gronwallprepare3} and obtain
\begin{equation}
\begin{split}
	& \| e_h^n \|_0^2 + \Delta t \sum_{i=1}^n \nu \| \til e_h^i \|_0^2 
	+ \frac{1}{2} \sum_{i=1}^n \| \til e_h^{i} - (i_h^T e_h^{i-1}) \circ X_1^{i-1} \|_0^2  
	+ \Delta t \sum_{i=1}^n \dzero  | \eps_h^i |_s^2 \\
	\leq & c( 1+\nu+\dzero+\dzero^{-1})( \| e_h^0 \|_0^2 + \Delta t^2 \| \nabla \eps_h^0 \|_0^2 + \Delta t^2 + h^{2k}).
\end{split}
	\label{eq:pfth3afterg}
\end{equation}
The estimate for the initial values are easily obtained from Hypotheses \ref{hypo:initSmallVis} and \ref{hypo:stab1}:
\begin{align*}
	\| e_h^0 \|_0 &\leq \| u_h^0 - u^0 \|_0 + \| u^0 - \appu_h^0 \|_0 \leq ch^k, \\
	\Delta t \| \nabla \eps_h^0 \|_0 & \leq \Delta t \| \nabla (p_h^0 - p^0) \|_0 + \Delta t \| \nabla (p^0 - \appp_h^0) \|_0 \leq c \Delta t.
\end{align*}

Now, the conclusion \eqref{eq:velocitySmallVis} follows from the triangle inequalities applied to 
$u_h - u = e_h - \eta$ and $\til u_h - u = \til e_h - \eta$,
\eqref{eq:pfth3afterg}, and 
\eqref{eq:hypoVQvinterpo}
for $\eta$.
We note that the estimate of $\|\til e_h^n\|_0$ follows from \eqref{eq:etilEst}, \eqref{eq:ehxEst} and \eqref{eq:pfth3afterg}.
\end{proof}

\section{An error estimate for the pressure}
\label{sec:errorEstiP}

In this section, to concentrate on the convergence order, 
we also use the notation $c_{\nu, \dzero}$ that may depend on $\nu, 1/\nu, \dzero$ and $1/\dzero$.
Additionally, we use notation and Lemmas in Subsections \ref{subsec:orderhypoandmaintheo} and \ref{subsec:prelV}.

\subsection{Hypotheses and the main theorem for the pressure}

\begin{hypo}\label{hypo:pestiReg}
The velocity $w$ and
the exact solution $(u, p)$ of the Oseen problem \eqref{Os} satisfy
\[
	w \in W^{2,\infty}(L^\infty) \cap H^1(W^{1,\infty}_0), \quad
	u \in Z^3 \cap H^2(H^{k+1}), \quad
	p \in H^2(H^k).
\]
\end{hypo}

We introduce the Stokes projection $(\pro u^*, \pro p^*) \in V_h \times Q_h$ of $(u^*, p^*) \in H^1_0(\Omega)^d \times L^2_0(\Omega)$, 
which satisfies the following equations:
\begin{subequations}
\begin{alignat}{3}
	a(\pro u^*, v_h) - (\pro p^*, \nabla \cdot  v_h) 
	&= a(u^*, v_h) -  (p^*, \nabla \cdot  v_h) \quad && \forall v_h \in V_h, \\
	-(\nabla \cdot \pro u^*, q_h) - \ch(\pro p^*, q_h) &= -(\nabla \cdot u^*, q_h) && \forall q_h \in Q_h.
\end{alignat}
\label{eq:StokesProj}
\end{subequations}

\begin{hypo}\label{hypo:initP}
The initial value $(u_h^0, p_h^0)$ satisfies 
$(u_h^0, p_h^0) = (\pro u^0, \pro p^0)$.
\end{hypo}

\begin{theorem}\label{thm:pres}
Let $(u_h, p_h):=\{(u_h^n, p_h^n)\}_{n=0}^{N_T}$ be the solution of 
Scheme$(k, k, \dzero)$ 
with $k \geq 1$ and $\dzero >0$.
Hypotheses \ref{hypo:pestiReg}, 
\ref{hypo:dt}, \ref{hypo:stab1}, and \ref{hypo:initP}.
Then the following estimate holds:
\begin{equation}
	\|p_h - p \|_{\ell^2(L^2)}
	\leq c_{\nu,\dzero}(\Delta t + h^k).
	\label{eq:prestheoconcl}
\end{equation}
\end{theorem}

\begin{remark}
For the initial value, 
\[
	\| u_h^0 - \pro u^0 \|_0, ~ \Delta t \| \nabla(p_h^0 - \pro p^0) \|_0
	\leq c_{\nu, \dzero} \Delta t (\Delta t + h^k)
\]
is actually needed in the proof of Theorem \ref{thm:pres} as in \cite{AchdouGuermond2000,Guermond1998}.
Hypothesis \ref{hypo:initP} is a sufficient condition.
\end{remark}

\subsection{Preliminaries for the pressure estimate} 

Lemma \ref{lemm:stabinfsup} is the weak inf-sup condition proved in \cite{BurmanFernandez2008}.

\begin{lemma} \label{lemm:stabinfsup}
Under Hypothesis \ref{hypo:stab1},
there exists a positive constant $c$ independent of $h$ such that
\[
	\|q_h\|_0 \leq c \sup_{v_h \in V_h \setminus\{0\}} \frac{(\nabla \cdot v_h,q_h)}{\|v_h\|_1} 
	+ c \szero(q_h, q_h)^{1/2},
	~ \forall q_h \in Q_h.
\]	
\end{lemma}

Lemma \ref{lemm:Stokesproj} is the error estimate of the Stokes projection in \eqref{eq:StokesProj}.
We omit the proof since it is the standard argument (e.g. \cite{FrancaStenberg1991,Garcia-archilla2020}).
We note that the constant $c_{\nu,\dzero}$ may depend on $\nu$ and $\dzero$.

\begin{lemma} \label{lemm:Stokesproj}
Let $V_h \times Q_h$ be the $\pk{k}/\pk{k}$-element for $k \geq 1$.
Suppose $(u^*, p^*) \in [H^{k+1}(\Omega)^d \cap H^1_0(\Omega)^d] \times [H^k(\Omega) \cap L^2_0(\Omega)]$.
Assume Hypothesis \ref{hypo:stab1} and $\dzero>0$.
Then, there exists a positive constant $c_{\nu,\dzero}$ independent of $h$ such that 
for any $h$ the Stokes projection $(\pro u^*, \pro p^*)$ of $(u^*, p^*)$
defined in \eqref{eq:StokesProj} satisfies
\begin{equation}
	\| \nabla (u^* - \pro u^*) \|_0, ~
	\| p^* - \pro p^* \|_0
	\leq c_{\nu, \dzero} h^{\ell} (\| u^* \|_{\ell+1} + \| p^* \|_{\ell}), \quad 1 \leq \ell \leq k.
	\label{eq:stokesprojRes1}
\end{equation}
\end{lemma}

Lemma \ref{lemm:twofunc} is a direct consequence of Lemma 4.5 in \cite{AchdouGuermond2000}, and 
Lemma \ref{lemm:bijectiveMix}-(\ref{item:jacobiest}) in this paper.
\begin{lemma}
\label{lemm:twofunc}
Let $1\leq q <\infty$, $1\leq p \leq \infty$, $1/p+1/p'=1$ and
$w_i\in W_0^{1,\infty}(\Omega)^d$, $i=1,2$.
Under the condition $\Delta t |w_i|_{1,\infty}\leq 1/4$,
it holds that,
for $v \in W^{1,qp'}(\Omega)^d$,
\begin{equation*}\label{eq:twofunc}
	\| v \circ X_1(w_1) - v \circ X_1(w_2)\|_{0,q} \leq 
	2^{1/(qp')}
	\Delta t \|w_1-w_2\|_{0,pq} 
	\|\nabla v \|_{0,qp'},
\end{equation*} 
where $X_1(\cdot)$ is defined in \eqref{eq:x1def}.
\end{lemma}

\begin{lemma} \label{lemm:DRlike}
Let $w_i \in W^{1,\infty}_0(\Omega)^d$ and $X_1(w_i)$ be the mapping defined in \eqref{eq:x1def}, $i=1, 2$.
Under the condition $\Delta t |w_i|_{1,\infty}\leq 1/4$,
there exists a positive constant $c$ independent of $\Delta t$
such that for $v \in L^2(\Omega)^d$ 
\begin{equation}
	\| v \circ X_1(w_1) - v \circ X_1 (w_2) \|_{-1} 
	\leq c \Delta t \| v \|_0 \| w_1 - w_2 \|_{1,\infty}.
	\label{eq:drlike}
\end{equation}
\end{lemma}

\begin{remark}
Lemma \ref{lemm:DRlike} is a generalization of Lemma 1 in \cite{DouglasRussell1982}.
When $w_1 = w$ and $w_2 = 0$,
\begin{equation}
	\| v \circ X_1(w) - v  \|_{-1} 
	\leq c \Delta t \| v \|_0 \| w\|_{1,\infty},
	\label{eq:DR}
\end{equation}
which is Lemma 1 in \cite{DouglasRussell1982}.
\end{remark}

\begin{proof}[Proof of Lemma \ref{lemm:DRlike}]
We denote $X_1(w_i)$ by $F_i$
and the Jacobian of $X_1(w_i)$ by $J_i$ for $i=1,2$,
which is positive 
because of Lemma \ref{lemm:bijectiveMix}-(\ref{item:jacobiest}).

In view of the definition
\begin{equation}
	\| v \circ F_1 - v \circ F_2 \|_{-1} 
	= \sup_{\Phi \in H^1_0(\Omega)^d \setminus \{0\} } \frac{(v \circ F_1 - v \circ F_2, \Phi)}{ \| \Phi \|_1 },
	\label{eq:h10justdef}
\end{equation}
we estimate $(v \circ F_1 - v \circ F_2, \Phi)$.
By the change of variable $y = F_i(x)$
and noting that $F_i: \Omega \to \Omega$ is bijective for $i=1,2$ (Lemma \ref{lemm:bijectiveMix}-(\ref{item:bijective})),
we have 
\begin{align}
	(v \circ F_1 - v\circ F_2, \Phi)
	&= \bigl( v, (\Phi \circ F_1^{-1}) J_1^{-1} - (\Phi \circ F_2^{-1}) J_2^{-1} \bigr) \notag \\
	& \leq \|v\|_0 \| (\Phi \circ F_1^{-1}) J_1^{-1} - (\Phi \circ F_2^{-1}) J_2^{-1} \|_0
	=: \|v\|_0 I_1. 
	\label{eq:prDRlike1}
\end{align}
The boundedness of the Jacobian (Lemma \ref{lemm:bijectiveMix}-(\ref{item:jacobiest})) 
yields
\begin{align}
	I_1 
	\leq & \| (\Phi \circ F_1^{-1}) J_1^{-1} - (\Phi \circ F_2^{-1}) J_1^{-1} \|_0 
		+ \| (\Phi \circ F_2^{-1}) (J_1^{-1} - J_2^{-1}) \|_0  \notag \\ 
	\leq &  \| \Phi \circ F_1^{-1} - \Phi \circ F_2^{-1} \|_0 \|J_1^{-1}\|_{0,\infty} + \| \Phi \circ F_2^{-1} \|_0 \| J_1^{-1} - J_2^{-1} \|_{0,\infty} \notag \\
	\leq & c \| \Phi \circ F_1^{-1}  - \Phi \circ F_2^{-1}  \|_0 
		+ c \| \Phi \|_0 \| J_1 - J_2 \|_{0,\infty}, 
		\label{eq:prDRlike2}
\end{align}
where we have used  
$J_1^{-1} - J_2^{-1} = J_1^{-1} J_2^{-1}(J_2 - J_1)$.
By the change of variable $x = F_2^{-1}(y)$,
\begin{align}
	& \| \Phi \circ F_1^{-1}  - \Phi \circ F_2^{-1} \|_0
	= \| \Phi \circ F_1^{-1} \circ F_2 \circ F_2^{-1}  - \Phi \circ F_1^{-1} \circ F_1 \circ F_2^{-1} \|_0
	\notag \\
	= & \bigl\| \bigl[ (\Phi \circ F_1^{-1}) \circ F_2  - (\Phi \circ F_1^{-1}) \circ F_1 \bigr] J_2^{1/2} \bigr\|_0 \notag \\
	\leq & c \Delta t \| \nabla(\Phi \circ F_1^{-1}) \|_0 \| w_1 - w_2 \|_{0,\infty},
	\label{eq:prDRlike3}
\end{align}
where we have used Lemma \ref{lemm:twofunc} with $q=2$, $p=\infty$, $p'=1$ and $v = \Phi \circ F_1^{-1}$.
We note that
\[
	|F_1(x_1) - F_1(x_2)| \geq |x_1 - x_2| - | w_1(x_1) - w_1(x_2) | \Delta t
	\geq (1-|w_1|_{1,\infty} \Delta t) |x_1 - x_2|,
\]
and $|w_1|_{1,\infty} \Delta t \leq 1/4$,
which implies
\[
	|F_1^{-1}(y_1) - F_1^{-1}(y_2)| \leq c |y_1 - y_2|
\]
and thus it holds that with the estimate of $J_1$ 
\begin{equation}
\begin{split}
	\bigl\| \nabla(\Phi \circ F_1^{-1}) \bigr\|_0 
	&= \bigl\| \bigl[ (\nabla \Phi) \circ F_1^{-1} \bigr] \nabla(F_1^{-1}) \bigr\|_0 \\ 
	&\leq \| (\nabla \Phi) \circ F_1^{-1} \|_0 \| \nabla (F_1^{-1}) \|_{0,\infty}
	\leq c \| \nabla \Phi \|_0.
\end{split}
	\label{eq:prDRlike4}
\end{equation}
We then have from \eqref{eq:prDRlike3} and \eqref{eq:prDRlike4}
\begin{equation}
	\| \Phi \circ F_1^{-1}  - \Phi \circ F_2^{-1} \|_0
	\leq c \Delta t \| \nabla \Phi \|_0 \| w_1 - w_2 \|_{0,\infty}.
	\label{eq:prDRlikeEsti1}
\end{equation}

From the definition of Jacobian $\det (\delta_{mn} - \partial w_m/ \partial x_n \Delta t)$, where $w=w_1$ or $w_2$,
and $\Delta t |w_1|_{1,\infty}$, $\Delta t |w_2|_{1,\infty} \leq 1/4$,
\begin{equation}
	\| J_1 - J_2 \|_{0,\infty} \leq c \Delta t \| w_1 - w_2 \|_{1,\infty}.
	\label{eq:prDRlikeJac}
\end{equation}

Now the conclusion \eqref{eq:drlike} follows
from 
\eqref{eq:h10justdef},
\eqref{eq:prDRlike1}, and \eqref{eq:prDRlike2} with
\eqref{eq:prDRlikeEsti1} and \eqref{eq:prDRlikeJac}.
\end{proof}

\subsection{Proof of Theorem \ref{thm:pres}}
Let $(z_h(t), r_h(t))$ be the Stokes projection $(\pro u(t), \pro p(t))$ of $(u(t), p(t))$ defined in \eqref{eq:StokesProj}.
We use the same notation in \eqref{eq:errornotation} after replacing $(z_h(t), r_h(t))$.
We note that the estimate
\begin{equation}
	\| e_h \|_{\ell^\infty(L^2)}, ~
	\| \nabla \til e_h \|_{\ell^2(L^2)}, ~
	|\eps_h|_{\ell^2(s)}
	\leq c_{\nu,\dzero} (\Delta t + h^k)
	\label{eq:ehproj}
\end{equation}
still holds for the new definition because, from Theorem \ref{thm:veloSmallVis} and Lemma \ref{lemm:Stokesproj},
\[
	\| e_h \|_{\ell^\infty(L^2)} 
	= \| u_h - \pro u \|_{\ell^\infty(L^2)} 
	\leq \| u_h - u \|_{\ell^\infty(L^2)}  + \| u - \pro u \|_{\ell^\infty(L^2)} 
	\leq c_{\nu,\dzero} (\Delta t + h^k).
\]
The estimate for $\| \nabla \til e_h \|_{\ell^2(L^2)}$ is done by the same way.
For $|\eps_h|_{\ell^2(s)}$,
from Hypothesis \ref{hypo:stab1}, Theorem \ref{thm:veloSmallVis} and Lemma \ref{lemm:Stokesproj},
with $\Pi_h$ being the interpolation operator in Hypothesis \ref{hypo:stab1},
\begin{align*}
	|\eps_h|_{\ell^2(s)} 
	&= |p_h - \pro p|_{\ell^2(s)} 
	\leq |p_h - \Pi_h p|_{\ell^2(s)} + |\Pi_h p - \pro p|_{\ell^2(s)} \\
	& \leq |p_h - \Pi_h p|_{\ell^2(s)} + c\|\Pi_h p - \pro p \|_{\ell^2(L^2)} \\
	&\leq |p_h - \Pi_h p|_{\ell^2(s)} + c\|\Pi_h p - p \|_{\ell^2(L^2)} + c\| p - \pro p \|_{\ell^2(L^2)}
	\leq c_{\nu,\dzero} (\Delta t + h^k).
\end{align*}

With new $(z_h^n, r_h^n)$, we have the following error equations for $n=0,1,...,N_T-1$ (cf.\eqref{eq:erroreq1}):
\begin{subequations}
\label{eq:erroreqSPro}
\begin{alignat}{2}
	\biggl( \frac{\til e_h^{n+1} - (i_h^T e_h^{n}) \circ X_1^n}{\Delta t}, v_h \biggr) + a(\til e_h^{n+1} ,v_h) + (\nabla \psi_h^{n}, v_h) 
	& = 
	( R_{11}^{n+1} + R_{12}^{n+1}, v_h ),
	~ &&\forall v_h \in V_h, 
	\label{subeq:erroreqSProa}\\
	\frac{e_h^{n+1} - \til e_h^{n+1}}{\Delta t} + \nabla(\eps_h^{n+1} - \psi_h^{n}) &= 0, \label{subeq:erroreqSProb}\\
	(e_h^{n+1}, \nabla q_h) - \ch(\eps_h^{n+1}, q_h)  &= 0, 
	~ &&\forall q_h \in Q_h, \label{subeq:erroreqSProc}
\end{alignat}
\end{subequations}
where 
$R_{11}^{n+1}$ and $R_{12}^{n+1}$ are defined in \eqref{eq:R11} and \eqref{eq:R12}, respectively.

Immediately we have from Lemma \ref{lemm:stabinfsup}
\begin{align}
	\| \eps_h^{n+1} \|_0 \leq &
	c \sup_{v_h \in V_h \setminus \{0\}} \frac{(\nabla \eps_h^{n+1}, v_h)}{\|v_h\|_1} 
	+ c \szero(\eps_h^{n+1}, \eps_h^{n+1})^{1/2}. \label{eq:epsInfsup}
\end{align}
For the estimate of $(\nabla \eps_h^{n+1}, v_h)/\|v_h\|_1$,
the following error equation is obtained from
\eqref{subeq:erroreqSProa} and \eqref{subeq:erroreqSProb}:
\begin{equation}
\begin{split}
	 \biggl( \frac{e_h^{n+1} - e_h^n}{\Delta t}, v_h \biggr)
		+ \biggl( \frac{i_h^T e_h^n - (i_h^T e_h^n) \circ X_1^n}{\Delta t}, v_h \biggr)
		+ a(\til e_h^{n+1}, v_h) + (\nabla \eps_h^{n+1}, v_h) \\
		= ( R_{11}^{n+1} + R_{12}^{n+1}, v_h ), \quad \forall v_h \in V_h.
\end{split}
\label{eq:erroreqEps}
\end{equation}
Here we note that $(i_h^T e_h^n, v_h) = (e_h^n, v_h)$ for $v_h \in V_h$.

The key is the estimate of $\| \frac{1}{\Delta t} (e_h^{n+1} - e_h^n) \|_{-1}$,
which is bounded by $L^2$-norm.
Let us use the notation
$d_t$ in \eqref{eq:dtdef} to get error equations for $d_t e_h^n$ and $d_t \varepsilon_h^n$.
In \eqref{eq:erroreqSPro}, 
we note that  
\[
	(i_h^T e_h^n) \circ X_1^n - (i_h^T e_h^{n-1}) \circ X_1^{n-1} = ( i_h^T d_t e_h^n) \circ X_1^n + (i_h^T e_h^{n-1}) \circ X_1^{n} - (i_h^T e_h^{n-1}) \circ X_1^{n-1}
\]
to obtain for $n=1,2,...,N_T-1$
\begin{subequations}
\label{eq:erroreq2}
\begin{align}
	\biggl( \frac{d_t \til e_h^{n+1} - (i_h^T d_t e_h^{n}) \circ X_1^n}{\Delta t}, v_h \biggr) 
	+ a( d_t \til e_h^{n+1} ,v_h) + (\nabla d_t \psi_h^{n}, v_h) 
	&= \bigl \langle R_2^{n+1}, v_h \bigr \rangle, 
		&& \forall v_h \in V_h, \label{subeq:erroreq2a} \\
	\frac{d_t e_h^{n+1}- d_t\til e_h^{n+1}}{\Delta t} + \nabla(d_t \eps_h^{n+1} - d_t \psi_h^{n}) &= 0, && \label{subeq:erroreq2b}\\
	(d_t e_h^{n+1}, \nabla q_h) - s_h(d_t \eps_h^{n+1}, q_h) &= 0, 
		&& \forall q_h \in Q_h, \label{subeq:erroreq2c}
\end{align}
\end{subequations}
where 
\begin{align}
	\bigl \langle R_2^{n+1}, v_h \bigr \rangle := & 
	\frac{1}{\Delta t} \left( (i_h^T e_h^{n-1}) \circ X_1^{n} - (i_h^T e_h^{n-1}) \circ X_1^{n-1}, v_h \right) 
	\notag \\ &
	 + (R_{11}^{n+1} - R_{11}^{n} + R_{12}^{n+1} - R_{12}^{n}, v_h) .
	 \label{eq:R2def}
\end{align}

The estimate for $\| R_{11}^n - R_{11}^{n-1} \|_0$ is found in \cite{AchdouGuermond2000} when the trajectory map is the solution of the ODE in \eqref{eq:xode}, and
we can obtain the same order for the Euler approximated map $X_1(\cdot)$ \cite{Misawa2016}.
We give a proof in Appendix \ref{subsec:proofoflemmadtR11R12} for completeness.
\begin{lemma} \label{lemm:dtR11R12}
Suppose that $w \in W^{2,\infty}(L^\infty) \cap C(W^{1,\infty}_0)$
and $\Delta t |w|_{C(W^{1,\infty})} \leq 1/4$.
Then, 
there exists a constant $c$ depending on 
the norm $\| w \|_{W^{2,\infty}(L^\infty)}$ 
such that
\begin{align}
	&\| R_{11}^n - R_{11}^{n-1} \|_0
	\leq 
	c
	\Delta t^{3/2} \| u \|_{Z^3(t^{n-2}, t^n)}, 
	\quad \forall u \in Z^3,
	\label{eq:dtR11} \\
	&\biggl\| \frac{v^n - v^{n-1} \circ X_1^{n-1}}{\Delta t} - \frac{v^{n-1} - v^{n-2} \circ X_1^{n-2}}{\Delta t} \biggr\|_0  \notag \\
	\leq &  
	c \Delta t^{1/2} 
	\| v \|_{H^2(t^{n-2}, t^n; L^2) \cap H^1(t^{n-2}, t^n; H^{1})} 
	+ c \Delta t \| v^{n-2} \|_{1}, \quad \forall v \in H^2(L^2) \cap H^1(H^1).
	\label{eq:dtR12}
\end{align}
\end{lemma}

\begin{lemma}\label{lemm:dtest1}
Assume Hypotheses \ref{hypo:pestiReg}, 
\ref{hypo:dt}, \ref{hypo:stab1}, and \ref{hypo:initP}.
Then the following estimate holds for $n=1,2,...,N_T$:
\begin{equation}
	\Bigl\| \frac{d_t e_h^n}{\Delta t}  \Bigr\|_0 
	\leq c_{\nu,\dzero} ( \Delta t + h^k ).
	\label{eq:dtest1}
\end{equation}
\end{lemma}

\begin{proof}
We apply Lemma \ref{lemm:errorEqArrange} to \eqref{eq:erroreq2} and obtain
\begin{equation}
\begin{split}
	& \frac{1}{2\Delta t}( \| d_t e_h^{n+1} \|_0^2 - \| (i_h^T d_t e_h^{n}) \circ X_1^n \|_0^2 + \| d_t \til e_h^{n+1} - (i_h^T d_t e_h^{n}) \circ X_1^n \|_0^2) + \nu \| \nabla d_t \til e_h^{n+1} \|_0^2  \\
	& 
	+ \frac{\Delta t}{2} (\| d_t \nabla \eps_h^{n+1} \|_0^2 -  \| d_t \nabla \psi_h^{n} \|_0^2) 
	+ \dzero \snorm{d_t \eps_h^{n+1}}^2 
	 = \bigl \langle R_2^{n+1}, d_t \til e_h^{n+1} \bigr \rangle
\end{split}
\label{eq:dterrorAfterStab}
\end{equation}
for $n=1,...,N_T-1$.

The first term in
$\langle R_2^{n+1}, d_t \til e_h^{n+1} \rangle$
is bounded
by Lemma \ref{lemm:DRlike} with $v = i_h^T e_h^{n-1}$, 
$w_1 = w^n$ and $w_2 = w^{n-1}$,
and \eqref{eq:dthalfest}:
\begin{equation*}
\begin{split}
	& \frac{1}{\Delta t} \left( (i_h^T e_h^{n-1}) \circ X_1^{n} - (i_h^T e_h^{n-1}) \circ X_1^{n-1}, d_t \til e_h^{n+1} \right) \\
	\leq & c \| i_h^T e_h^{n-1} \|_0 \| w^n - w^{n-1} \|_{1,\infty} \| d_t \til e_h^{n+1} \|_1 \\ 
	\leq &  \frac{c \Delta t}{\nu} \|e_h^{n-1}\|_0^2 \biggl\| \frac{\partial w}{\partial t} \biggl\|_{L^2(t^{n-1}, t^n; W^{1,\infty})}^2
	+ \nu \| \nabla d_t \til e_h^{n+1} \|_0^2.
	\label{eq:inverseEst1}
\end{split}
\end{equation*}
Other terms in $\langle R_2^{n+1}, d_t \til e_h^{n+1} \rangle$ can be estimated, 
as in \eqref{eq:R3Schwarz},
by
\begin{equation*}
	|(R^{n+1}_{11} - R^{n}_{11} + R^{n+1}_{12} - R^{n}_{12}, d_t \til e_h^{n+1})| 
	\leq 
	\frac{1}{\gamma_0} \|R^{n+1}_{11} - R^{n}_{11} \|_0^2
	+ \frac{1}{\gamma_0} \| R^{n+1}_{12} - R^{n}_{12} \|_0^2 
	+ \frac{\gamma_0}{2} \| d_t \til e_h^{n+1}\|_0^2
\end{equation*}
and by Lemma \ref{lemm:dtR11R12}
with $v=\eta$.
Here $\gamma_0$ is chosen so that $\frac{1}{2\Delta t} \geq \frac{1}{2\Delta t_0} \geq \gamma_0$.
As in \eqref{eq:etilEst} and \eqref{eq:ehxEst}, we also use the inequalities
\begin{equation*}
	\frac{\gamma_0}{2} \| d_t \til e_h^{n+1} \|_0^2 \leq  
	\gamma_0 \| d_t \til e_h^{n+1} - (i_h^T d_t e_h^n) \circ X_1^n \|_0^2
	+ \gamma_0 \| (i_h^T d_t e_h^n) \circ X_1^n \|_0^2,
\end{equation*}
\begin{equation*}
	\| (i_h^T d_t e_h^n) \circ X_1^n \|_0^2 
	\leq (1+c\Delta t) \| i_h^T d_t e_h^n \|_0^2
	\leq (1+c\Delta t) \| d_t e_h^n \|_0^2.
\end{equation*}
The estimate for $\| d_t \nabla \psi_h^n \|_0^2$ is obtained by \eqref{eq:dtsechalfest} as follows:
\begin{align*}
	\| d_t \nabla \psi_h^n \|_0^2
	& \leq (1+\Delta t) \| d_t \nabla \eps_h^n \|_0^2 + \Bigl( 1+\frac{1}{\Delta t} \Bigr) \|\nabla d_t r_h^{n+1} - \nabla d_t r_h^n\|_0^2 \\
	& \leq (1+\Delta t) \| d_t \nabla \eps_h^n \|_0^2 
	+ c \Bigl( 1+\frac{1}{\Delta t} \Bigr) \Delta t^3 \|r_h\|_{H^2(t^{n-1},t^{n+1}; H^1)}^2.
\end{align*}

Gathering these estimates, 
from \eqref{eq:dterrorAfterStab},
we now obtain
for $n=1,2,...,N_T-1$, 
\[
	\frac{x^{n+1} - x^n}{\Delta t} 
	\leq c x^n + c b^{n+1},
\]
where 
\begin{align*}
	x^n = & \| d_t e_h^n \|_0^2 + \Delta t^2 \| d_t \nabla \eps_h^n \|_0^2, \\
	b^n = & \frac{\Delta t}{\nu} \|e_h\|_{\ell^\infty(L^2)}^2 \| w \|_{H^1(t^{n-2}, t^{n-1}; W^{1,\infty})}^2
	+ \Delta t^{3} \Bigl( \| u \|_{Z^3(t^{n-2}, t^n)}^2 + \|r_h\|_{H^2(t^{n-2},t^{n}; H^1)}^2 \Bigr) \\
		& + \Delta t \| \eta \|_{H^2(t^{n-2}, t^n; L^2) \cap H^1(t^{n-2}, t^n; H^{1})}^2 
		+ \Delta t^2 \| \eta^{n-2} \|_1.
\end{align*}
Using Gronwall's inequality (Lemma \ref{lem:discreteGronwall}) with $n_0=1$,
Lemma \ref{lemm:Stokesproj} for $(\eta, r_h)$ and \eqref{eq:ehproj} for $e_h$,
we get
for $n=1,..., N_T$
\begin{equation}
	x^n \leq c_{\nu,\dzero} \bigl[ x^1 + \Delta t^2 (\Delta t^2 + h^{2k}) \bigr].
	\label{eq:dtestiafterGron}
\end{equation}

Finally we estimate $x^1$.
Since $e_h^0=0$, and $\eps_h^0 = 0$ from Hypothesis \ref{hypo:initP}, 
\begin{equation}
	x^1 = \| e_h^1 \|_0^2 + \Delta t^2 \| \nabla \eps_h^1 \|_0^2.
	\label{eq:dtinitx1expres}
\end{equation}
Again, $e_h^0=0$,
\eqref{eq:erroreqSPro} with $n=0$ and Lemma \ref{lemm:errorEqArrange} yields
\begin{equation}
	\frac{1}{2\Delta t} (\| e_h^1 \|_0^2 + \| \til e_h^1 \|_0^2 ) + \nu \| \nabla \til e_h^1 \|_0^2 
	+ \dzero \snorm{\eps_h^1}^2
	+ \frac{\Delta t}{2} ( \| \nabla \eps_h^1 \|_0^2 - \| \nabla \psi_h^0 \|_0^2 )
	= ( R_{11}^1 + R_{12}^1, \til e_h^1 ).
\label{eq:step1Afterstab}
\end{equation}
By $\eps_h^0 = 0$ and \eqref{eq:dthalfest},
\begin{align*}
	\| \nabla \psi_h^0 \|_0^2
	& \leq 2 \|\nabla \eps_h^0 \|_0^2 + 2 \|\nabla(r_h^{1} - r_h^0)\|_0^2 
	\leq 2 \Delta t \| p \|_{H^1(t^0, t^{1}; H^1)}^2
	\leq c \Delta t^2 \| p \|_{C^1(H^1)}^2.
\end{align*}
For the right hand side
\begin{equation*}
	|( R_{11}^1 + R_{12}^1, \til e_h^1 )|
	\leq \Delta t (\| R_{11}^1 \|_0^2 + \| R_{12}^1 \|_0^2 ) + \frac{1}{2 \Delta t} \| \til e_h^1 \|_0^2.
\end{equation*}
The estimate of the first and second term are obtained by Lemma \ref{lemm:estR1} with $v=\eta$,
and the last term is absorbed by the left hand side of \eqref{eq:step1Afterstab}.
We then have
\begin{equation}
	\| e_h^1 \|_0^2 + \Delta t^2 \| \nabla \eps_h^1 \|_0^2 
	\leq c_{\nu,\dzero} \Delta t^2 (\Delta t^2 + h^{2k}).
	\label{eq:dtInitEsti}
\end{equation}

Now, the conclusion \eqref{eq:dtest1} follows from \eqref{eq:dtestiafterGron}, \eqref{eq:dtinitx1expres} and \eqref{eq:dtInitEsti}.
\end{proof}

\begin{proof}[Proof of Theorem \ref{thm:pres}]
The inequality \eqref{eq:DR} with $v=i_h^T e_h^n$ and $w=w^{n}$ yields
\[
	\biggl\| \frac{i_h^T e_h^n - (i_h^T e_h^n) \circ X_1^n}{\Delta t} \biggr\|_{-1}
	\leq c \| i_h^T e_h^n \|_0 \|w^n\|_{1,\infty} 
	\leq c \| e_h^n \|_0.
\]
From 
\eqref{eq:erroreqEps}, the inequality $\|\cdot \|_{-1} \leq \|\cdot\|_0$,
Lemma \ref{lemm:dtest1},
and Lemma \ref{lemm:estR1} with $v=\eta$
\begin{align*}
	\sup_{v_h \in V_h \setminus \{0\}} \frac{(\nabla \eps_h^{n+1}, v_h)}{\|v_h\|_1}
	\leq & \biggl\| \frac{e_h^{n+1} - e_h^n}{\Delta t} \biggr\|_{-1}
	+ \biggl\| \frac{i_h^T e_h^n - (i_h^T e_h^n) \circ X_1^n}{\Delta t} \biggr\|_{-1} \\
	&+ \nu \| \nabla \til e_h^{n+1} \|_0
	+ \|R_{11}^{n+1}\|_{-1} + \|R_{12}^{n+1}\|_{-1} \\
	\leq & c_{\nu,\dzero} (\Delta t + h^k) 
	+ c \| e_h^n \|_0
	+ \nu \| \nabla \til e_h^{n+1} \|_0
	+ \| \eta \|_{H^1(t^{n}, t^{n+1}; L^2) \cap L^2(t^{n}, t^{n+1}; H^{1})}.
\end{align*}
The estimate for $\eta$ is obtained by Lemma \ref{lemm:Stokesproj}.
Then, from \eqref{eq:epsInfsup} with the estimate above,
we have
\begin{align*}
	\| \eps_h \|_{\ell^2(L^2)}
	\leq c_{\nu,\dzero} (\Delta t + h^{k}) 
		+ c \|e_h \|_{\ell^\infty(L^2)}
		+ c \nu \| \nabla \til e_h \|_{\ell^2(L^2)}
		+ c |\eps_h|_{\ell^2(s)}.
\end{align*}
The estimates for $\|e_h \|_{\ell^\infty(L^2)}$, $\| \nabla \til e_h \|_{\ell^2(L^2)}$ and $|\eps_h|_{\ell^2(s)}$ are obtained by \eqref{eq:ehproj}.
Now the conclusion \eqref{eq:prestheoconcl} follows from the triangle inequality applied to 
$p_h-p = \eps_h - (p-\pro p)$ and Lemma \ref{lemm:Stokesproj}.
\end{proof}

\section{Numerical results} 
\label{sec:numericalResults}
We compare the result of Scheme(2,1,0) (Taylor--Hood element) to Scheme($k$,$k$,$\dzero$) with $k=1,2$ and $\dzero>0$.
We use the stabilization term $\szero$ in \eqref{eq:stabBPB}.
We implement the practical scheme \eqref{eq:schemePracProj}--\eqref{eq:schemePracb}.
We integrate the term 
\eqref{eq:intLLV} exactly 
instead of the original one $((i_h^T u_h^n) \circ X_1(w^n), v_h)$ in \eqref{eq:schemePraca}.

Let $\Omega = (0,1)^2$, $T=1$.
The functions $f$ and $u^0$ are defined so that the exact solution is
\begin{align*}
	u_1(x,t) &= (1 + \sin(\pi t)) \sin(\pi x_1)^2 \sin(2 \pi x_2), \\
	u_2(x,t) &= -(1 + \sin(\pi t)) \sin(2\pi x_1) \sin(\pi x_2)^2,\\
	p(x,t) & = - \cos(\pi x_2) + \frac{1}{2} \cos(4 \pi(t + x_1)) .
\end{align*}
The velocity $w$ is also set to be $u$.

FreeFEM \cite{FreeFemCite} is used only for triangulations of the domain.
Let $N=16, 23, 32 ,45$ and $64$ be the division number of each side of $\overline \Omega$, and
we set $h = 1/N$. 
Figure \ref{fig:irsq_renum} shows the triangulation of $\overline \Omega$ when $N=16$. 
The time increment $\Delta t$ is set to be $\Delta t=h^2$ for Scheme(2,1,0) and Scheme(2,2,$\dzero$), 
and $\Delta t=(1/16)h$ for Scheme(1,1,$\dzero$)
to observe the convergence behavior.
This choice is not based on the stability condition.

\begin{figure}
	\centering
	\includegraphics[width=1.5in]{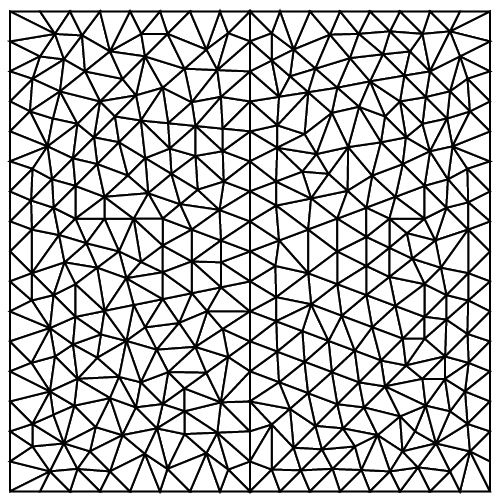}
	\caption{The triangulation of $\overline \Omega$ for $N=16$.
	}
	\label{fig:irsq_renum}
\end{figure}

The initial value $(u_h^0, p_h^0)$ is set to be the Lagrange interpolation of $(u^0, p^0)$ in $\mathrm P_k/\mathrm P_\ell$-element space
for Scheme$(k,\ell,\dzero)$.

\begin{remark}
Lagrange interpolation is sufficient for Hypothesis \ref{hypo:initSmallVis} in the velocity estimate (Section \ref{sec:errorEstiSmallVis}) but not for Hypothesis \ref{hypo:initP} in the pressure estimate (Section \ref{sec:errorEstiP}).
For the effect to the pressure solution at the first step, see \cite{BurmanFernandez2008,JohnNovo2015}.
\end{remark}

Recall the norm notation in \eqref{eq:discNorms}.
The relative error $E_X$ is defined by
\begin{equation*}
	E_X(u) = \frac{\| u-\til u_h \|_{X,h}}{\| u\|_{X,h}}, \quad
	E_X(p) = \frac{\| p-p_h \|_{X,h}}{\| p\|_{X,h}},
\end{equation*}
where $X=\ell^\infty(L^2)$ or $\ell^2(H_0^1)$ for $u$, 
$X=\ell^2(L^2)$ for $p$,
and $\| \cdot \|_{X,h}$ means that the spatial norm is computed approximately by numerical quadrature of order nine \cite{9thArt}.
Table \ref{table:symbol} shows the symbols used in graphs. 
Since every graph of the relative error $E_X$ versus $h$ is depicted in the logarithmic scale, the slope corresponds to the convergence order. 

\begin{table}
	\caption{Symbols used in the graphs.}
	\centering
	\begin{tabular}{cccc}
		\hline
		$\phi$ 				& $u$ 	& $u$ 	& $p$ \\
		$X$ 				& $\ell^\infty(L^2)$& $\ell^2(H^1_0)$ & $\ell^2(L^2)$ \\
		\hline
		\hline
		Scheme($2,1,0$)	 	& \mone & \mtwo & \mthree \\
		Scheme($k,k,\dzero$)	& \mfour& \mfive& \msix \\
		\hline
	\end{tabular}
	\label{table:symbol}
\end{table}

\begin{figure}
\centering
\includegraphics[height=3in]{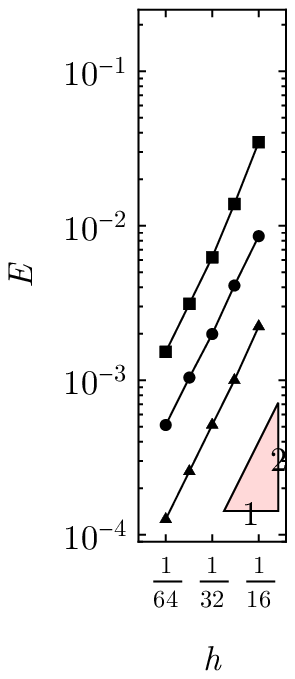} \qquad
\includegraphics[height=3in]{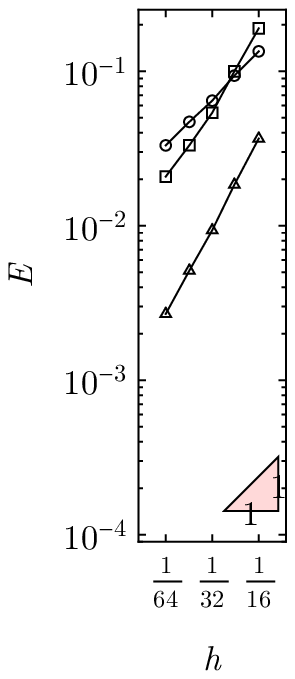} \qquad
\includegraphics[height=3in]{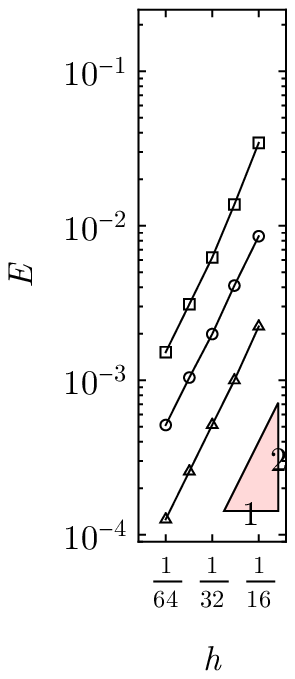}
\caption{Relative errors versus $h$ for $\nu=1$. 
Scheme(2,1,0) with $\Delta t=h^2$ (left), Scheme(1,1,$10^{-1}$)  with $\Delta t=(1/16)h$ (center), Scheme(2,2,$10^{-2}$) with $\Delta t=h^2$ (right).
}
\label{fig:nu0HEgraph}
\end{figure}

\begin{figure}
\centering
\includegraphics[height=3in]{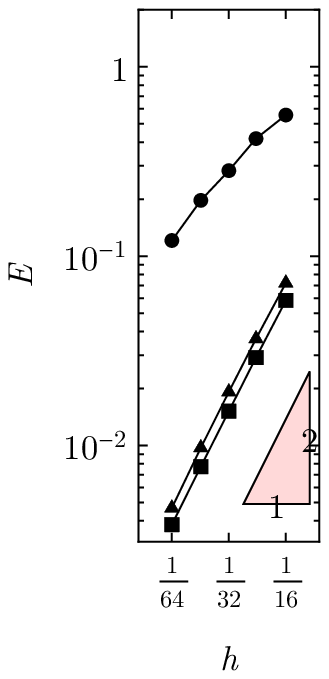} \qquad
\includegraphics[height=3in]{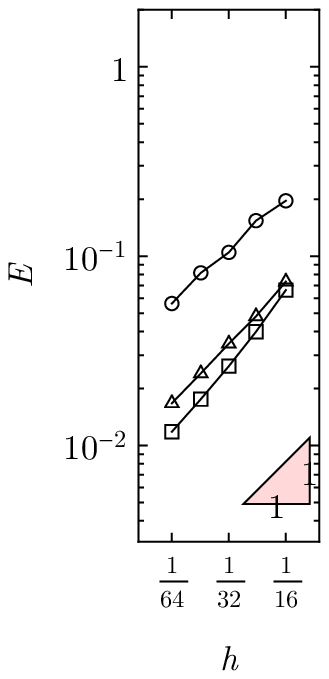} \qquad
\includegraphics[height=3in]{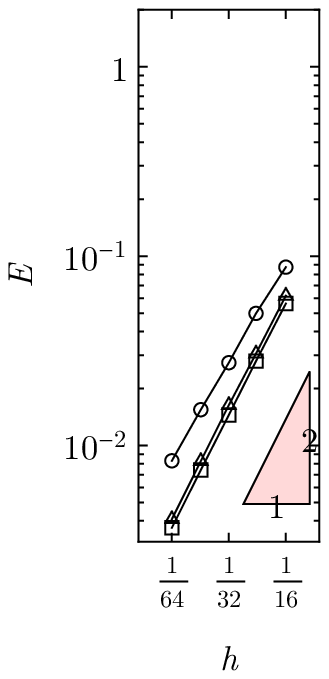}
\caption{Relative errors versus $h$ for $\nu=10^{-4}$. 
Scheme(2,1,0) with $\Delta t=h^2$ (left), Scheme(1,1,$10^{-1}$)  with $\Delta t=(1/16)h$ (center), Scheme(2,2,$10^{-2}$) with $\Delta t=h^2$ (right).
}
\label{fig:nu4HEgraph}
\end{figure}

Figure \ref{fig:nu0HEgraph} shows the graphs of the errors versus $h$ when $\nu=1$.
For Scheme(2,1,0) and Scheme(2,2,$10^{-2}$), all convergence orders are almost two with no significant differences.
For Scheme(1,1,$10^{-1}$), the convergence orders of $E_{\ell^\infty(L^2)}(u)$ (\mfour) and $E_{\ell^2(L^2)}(p)$ (\msix) are greater than one.
These exceed prediction from the theoretical result.
Figure \ref{fig:nu4HEgraph} shows the graphs when $\nu=10^{-4}$.
For Scheme(2,1,0) and Scheme(2,2,$10^{-2}$),
there are no significant difference in $E_{\ell^\infty(L^2)}(u)$ (\mone,\mfour) and $E_{\ell^2(L^2)}(p)$ (\mthree,\msix).
In Scheme(2,1,0), meanwhile, convergence order of $E_{\ell^2(H^1_0)}(u)$ (\mtwo) is about 0.8 to 1.4, which is less than 2.
In Scheme(2,2,$10^{-2}$), convergence order of $E_{\ell^2(H^1_0)}(u)$ (\mfive) is about 1.5 to 1.8.
To observe the convergence order $O(h^2)$, finer meshes will be necessary.
The error of Scheme(2,2,$10^{-2}$) (\mfive) is almost ten times less than that of Scheme(2,1,0) (\mtwo) for $h=1/64$.
We also observe that the errors $E_{\ell^2(H^1_0)}(u)$ of Scheme(1,1,$10^{-1}$) (\mfive) is less than that of Scheme(2,1,0) (\mtwo).

\begin{figure}
\centering
\includegraphics[width=1.5in]{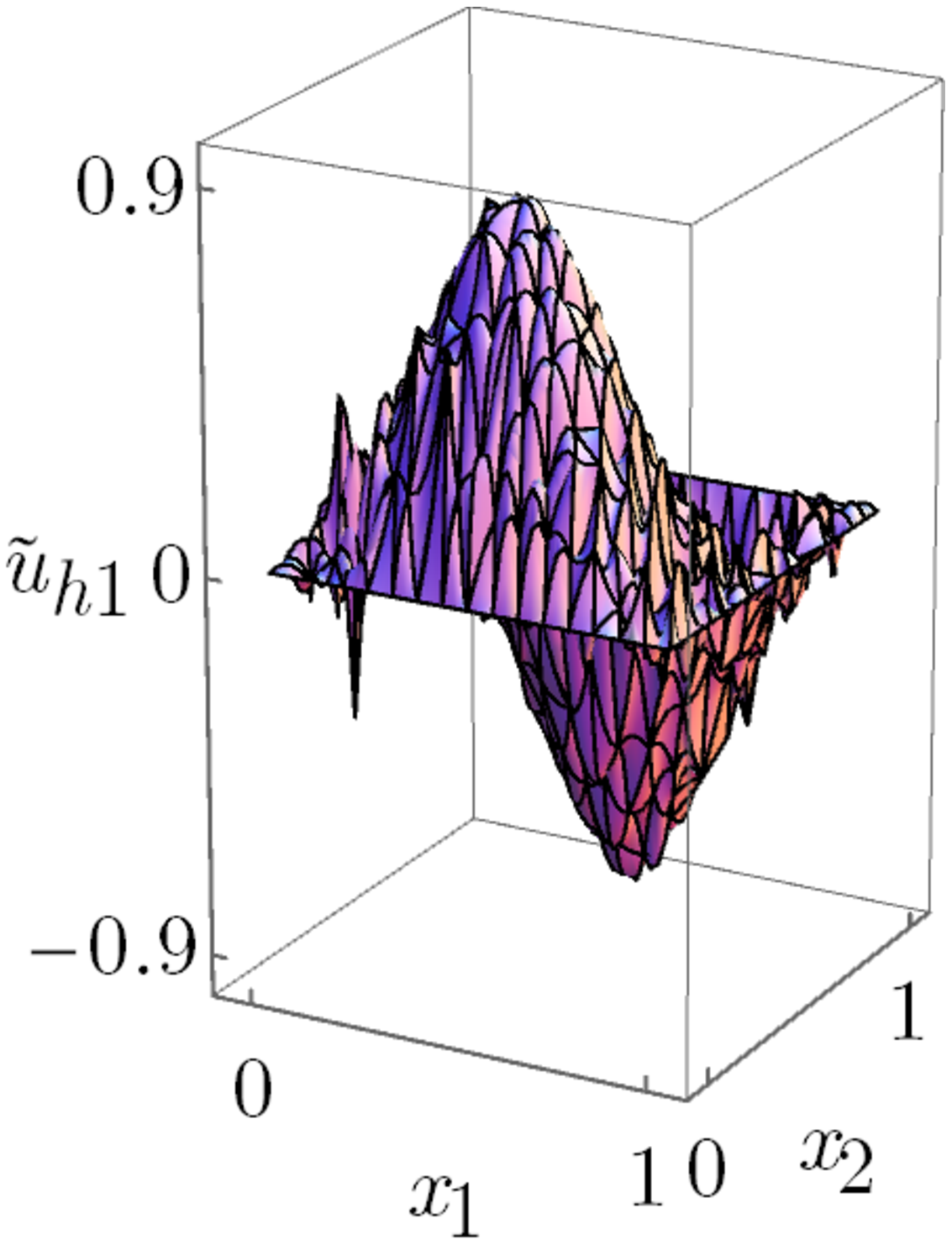}  \quad
\includegraphics[width=1.5in]{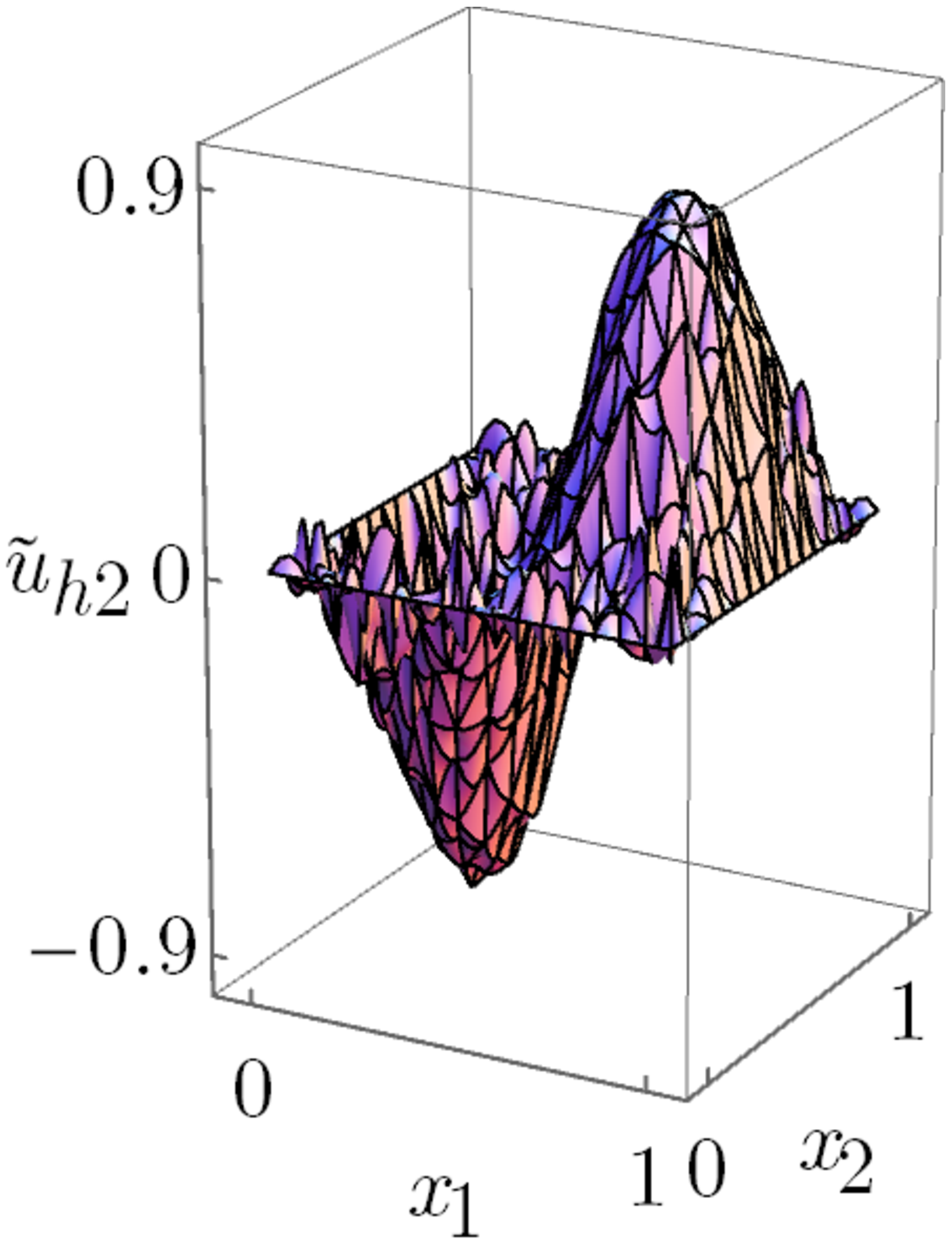} \quad
\includegraphics[width=1.5in]{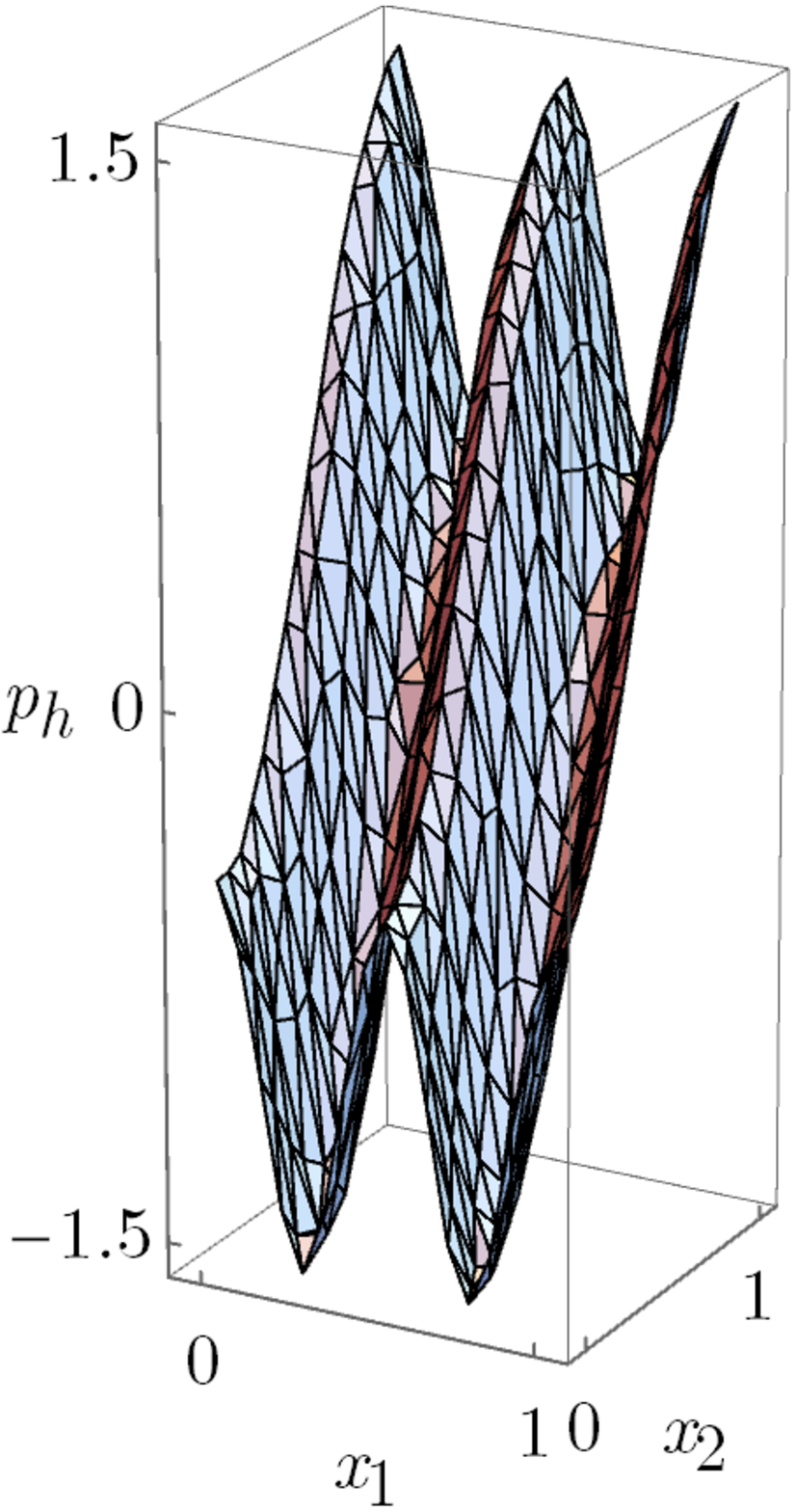}
\caption{Stereographs of $\til u_{h1}^n, \til u_{h2}^n$ and $p_{h}^n$ of Scheme(2,1,0), $\nu=10^{-4}$, $t^n=1$, $h=1/16$, $\Delta t=h^2$.}
\label{fig:stereo_p2p1}
\end{figure}

\begin{figure}
\centering
\includegraphics[width=1.5in]{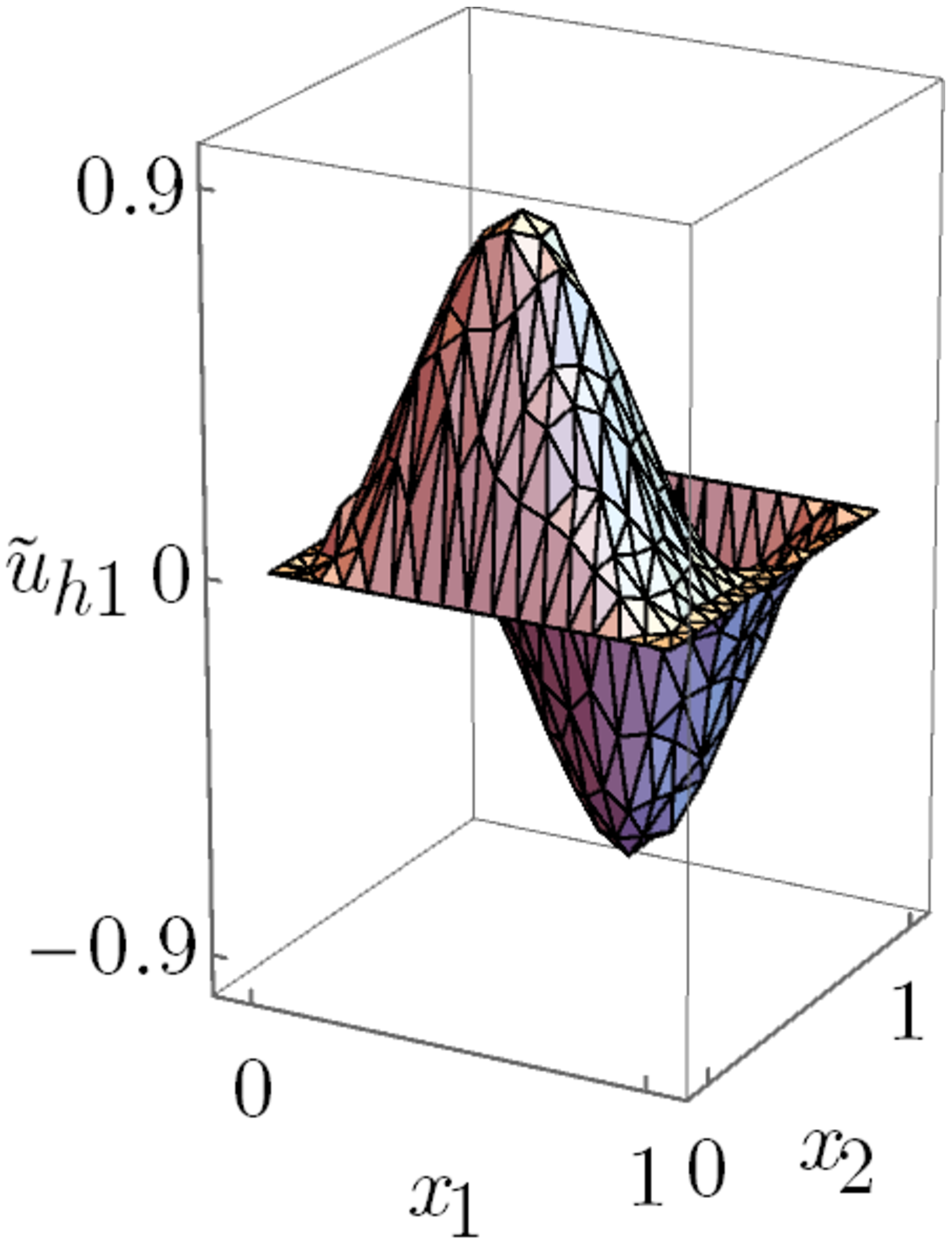}  \quad
\includegraphics[width=1.5in]{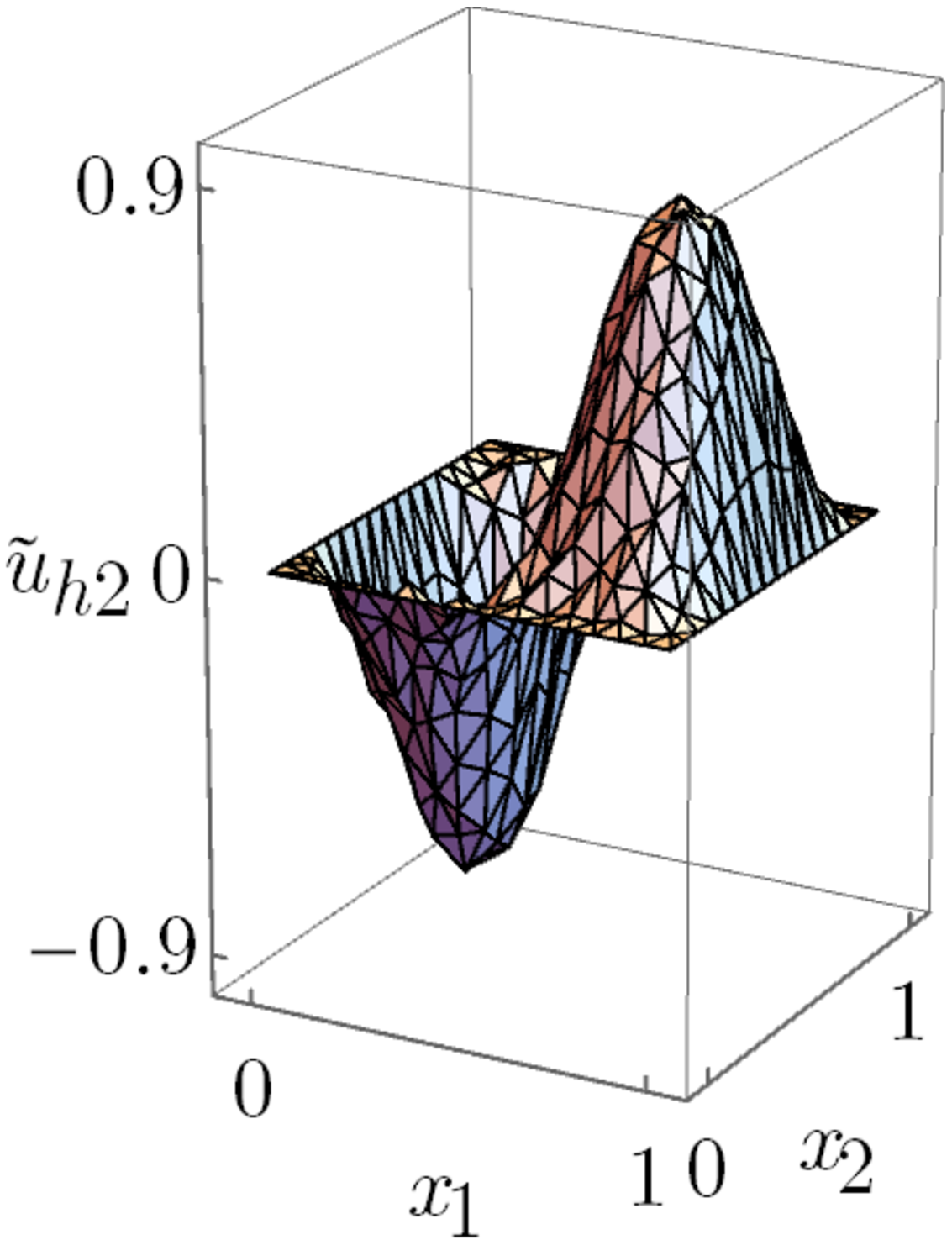} \quad
\includegraphics[width=1.5in]{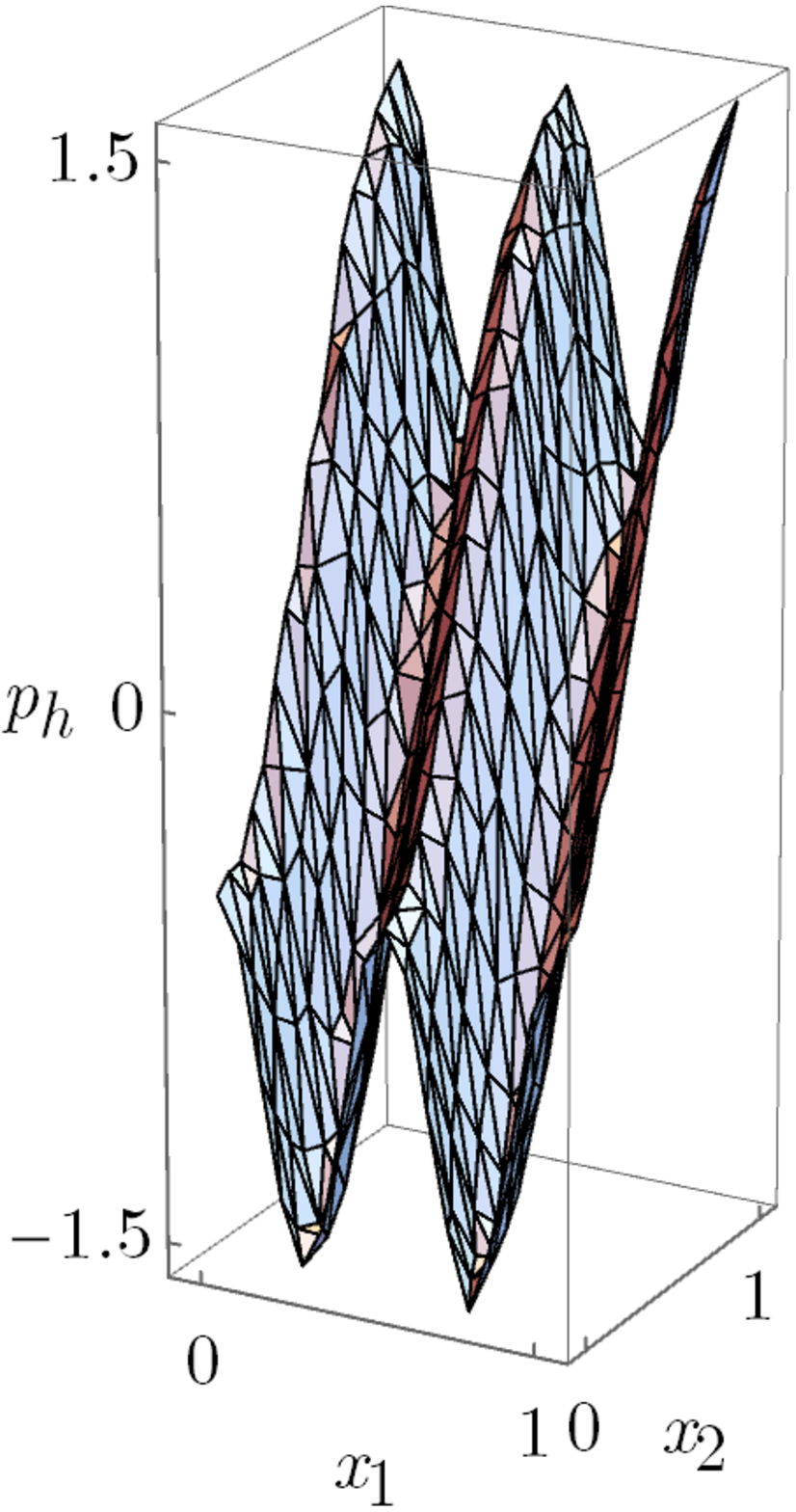}
\caption{Stereographs of $\til u_{h1}^n, \til u_{h2}^n$ and $p_{h}^n$ of Scheme(1,1,$10^{-1}$), $\nu=10^{-4}$, $t^n=1$, $h=1/16$, $\Delta t=(1/16)h$.}
\label{fig:stereo_p1p1}
\end{figure}

\begin{figure}
\centering
\includegraphics[width=1.5in]{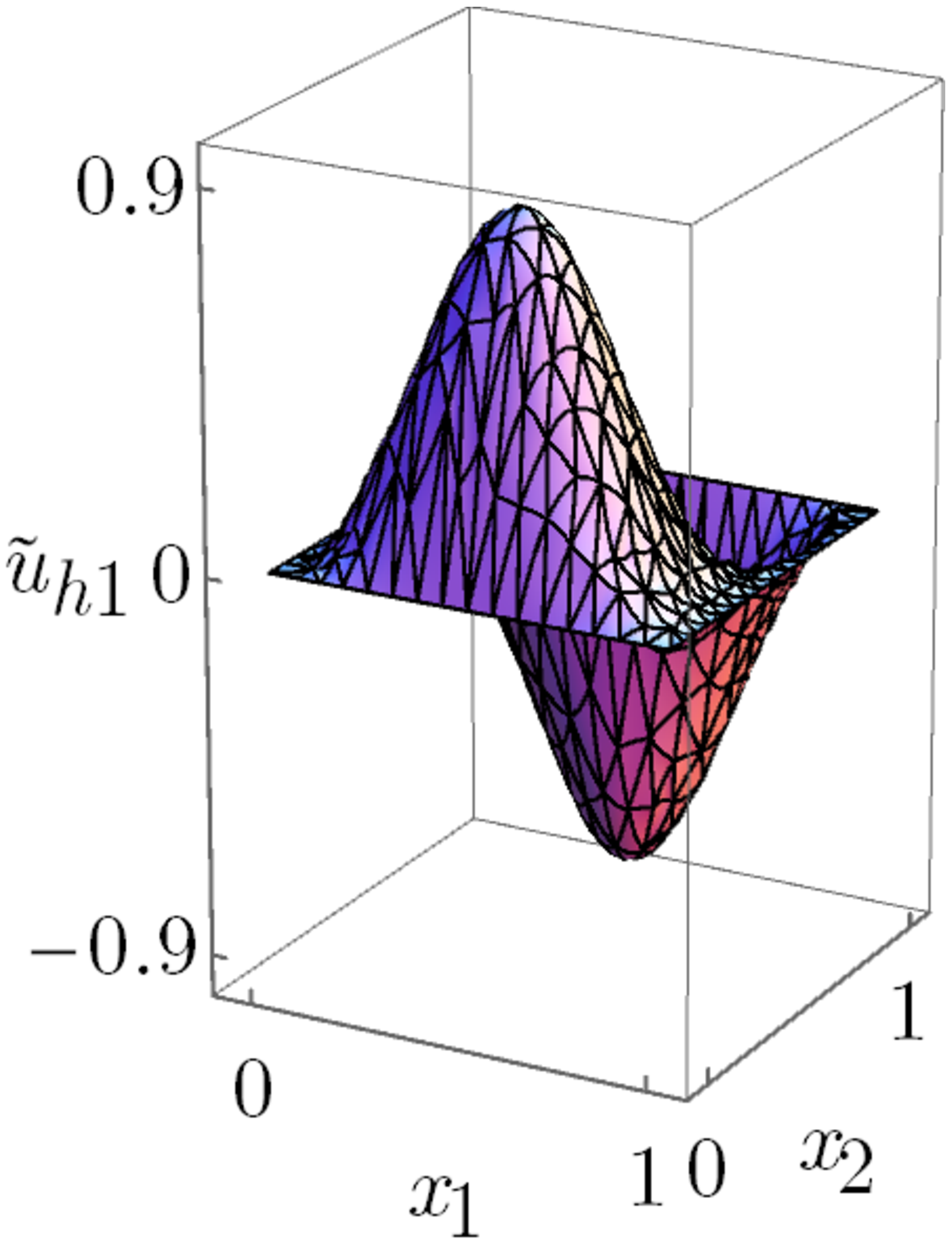}  \quad
\includegraphics[width=1.5in]{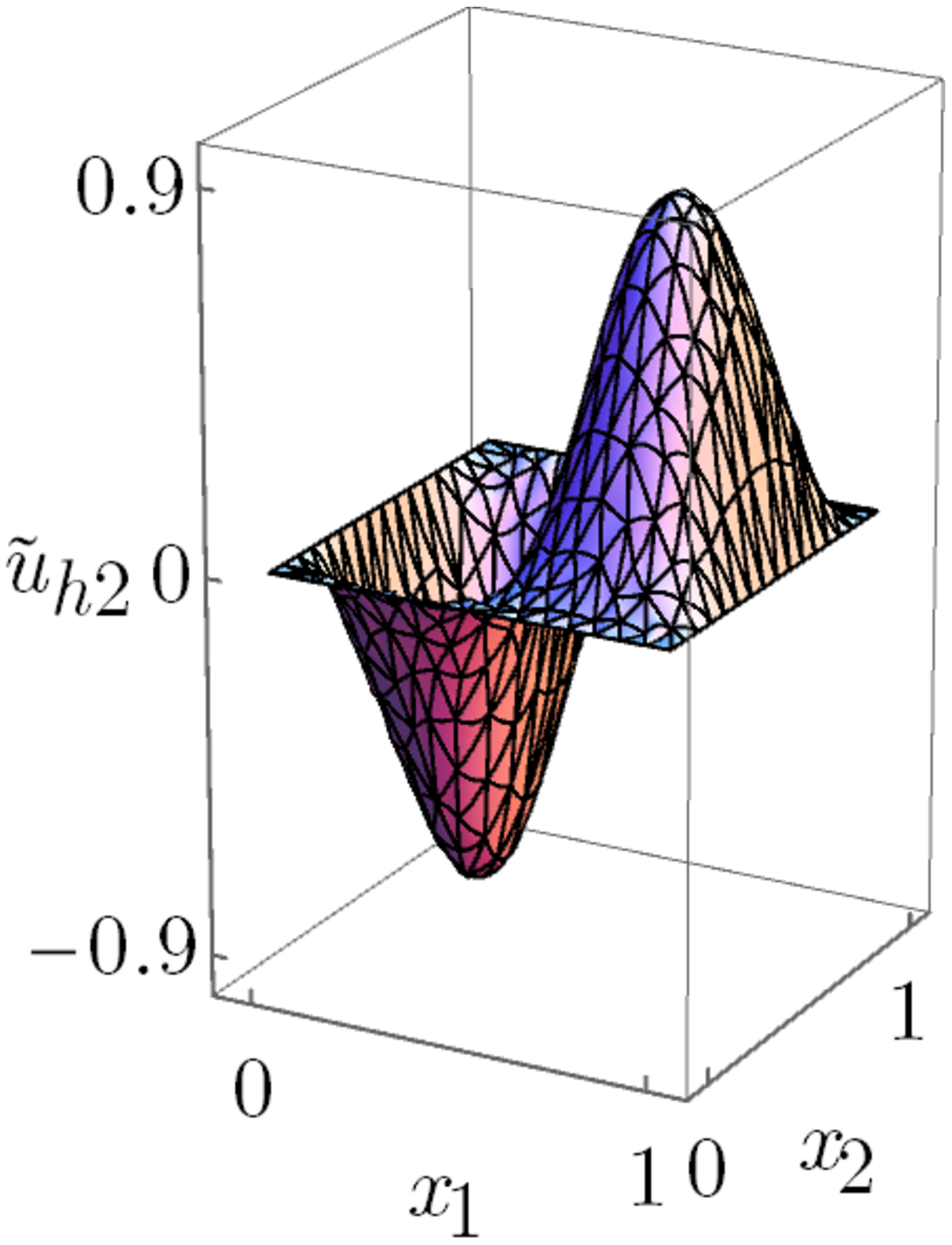} \quad
\includegraphics[width=1.5in]{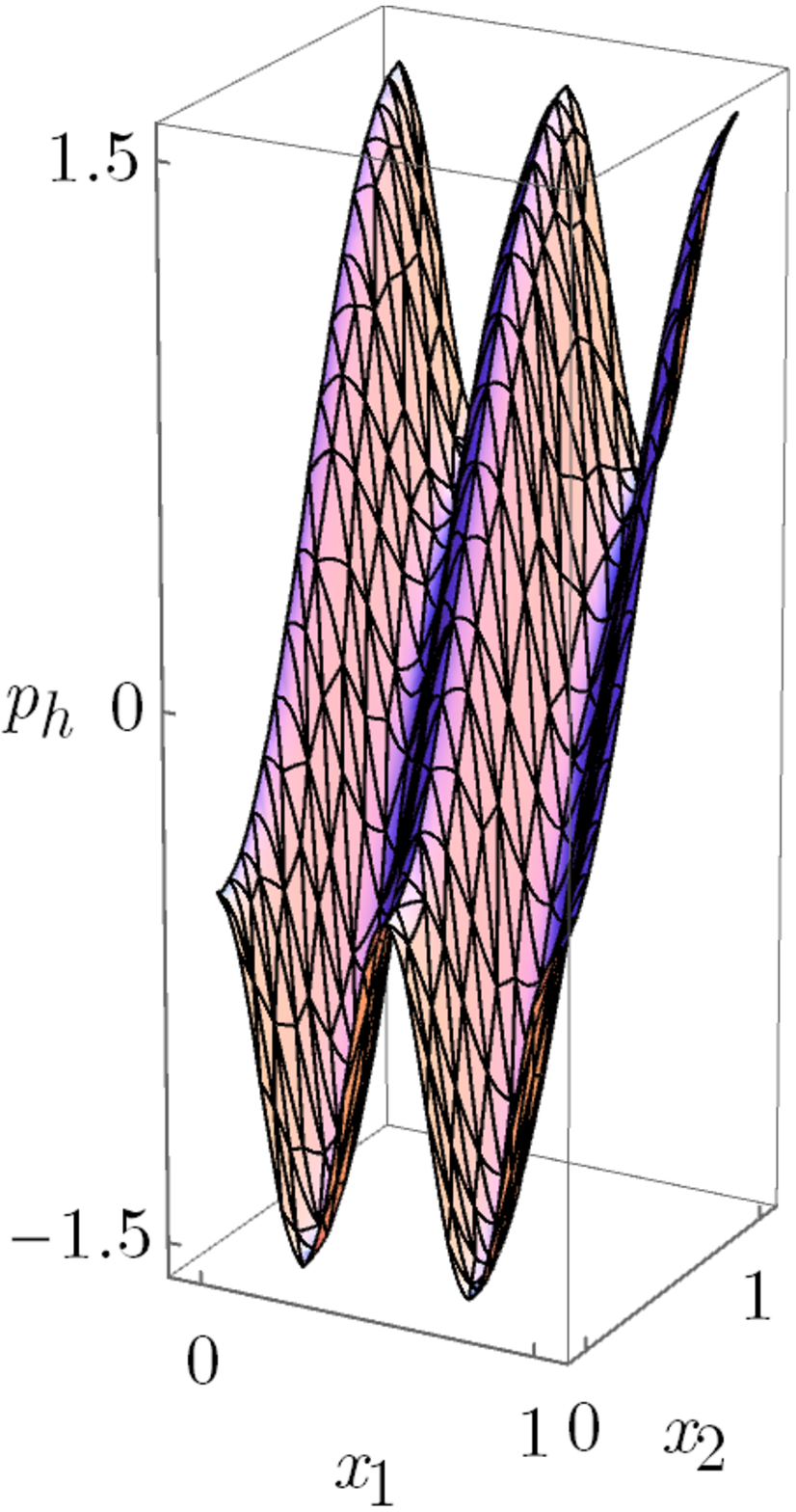}
\caption{Stereographs of $\til u_{h1}^n, \til u_{h2}^n$ and $p_{h}^n$ of Scheme(2,2,$10^{-2}$), $\nu=10^{-4}$, $t^n=1$, $h=1/16$, $\Delta t=h^2$.}
\label{fig:stereo_p2p2}
\end{figure}

Figures \ref{fig:stereo_p2p1}--\ref{fig:stereo_p2p2} show the stereographs of $\til u_{h1}^n$, $\til u_{h2}^n$ and $p_{h}^n$ when $\nu=10^{-4}$, $t^n=1$, $h=1/16$, and $\Delta t=1/256$.
In Figure \ref{fig:stereo_p2p1},
we observe unnatural oscillation in the velocity of Scheme(2,1,0),
which corresponds to the large error $E_{\ell^2(H^1_0)}(u)$ (\mtwo) in Figure \ref{fig:nu4HEgraph}.
For the solutions of Scheme(1,1,$10^{-1}$) and  Scheme(2,2,$10^{-2}$),
we have no significant oscillation.

\section{Concluding remarks}
\label{sec:conclusions}

We developed and analyzed a pressure-stabilized projection LG scheme for the Oseen problem.
The scheme inherits the advantages of computational efficiency from the projection/LG combined scheme.
Here, since the viscosity term in the Oseen equations is Laplacian, the matrices can be decoupled into each component of the velocity and the pressure.
We used the equal-order pair of the finite element for the velocity and pressure.
Approximability of the pressure space is actually used in the velocity error estimate for small viscosity.
We also derived the pressure error estimate, where the constant depends on the viscosity. 
Numerical results showed higher accuracy of the equal-order element than the Taylor--Hood element for small viscosity.

\appendix
\section{Appendix}

\subsection{Proof of Lemma \ref{lemm:dtR11R12}}
\label{subsec:proofoflemmadtR11R12}

\begin{proof}
We prove \eqref{eq:dtR11}.
We apply Taylor's theorem
\[
	g(1) = g(0) + g'(0) + \frac{1}{2} g''(0) + \int_0^1 \frac{(1-s)^2}{2} g'''(s) ds
\]
to $g(s) = u(y_j(x,s), t_j(s))$ 
for $j=1,2$,
where 
\[
	(y_j(x,s), t_j(s)) = (x - s w^{n-j}(x) \Delta t, t^{n-j+1} - s \Delta t)\]
so that $g(1) = (u^{n-j} \circ X_1(w^{n-j}))(x)$,
$g(0)=u^{n-j+1}$, $g'(0) = -\Delta t D_w^{n-j} u^{n-j+1}$
and $g''(0) = \Delta t^2 (D_w^{n-j})^2 u^{n-j+1}$.
Here we have used the following material derivative
\[
	D_w^{n} := \frac{\partial}{\partial t} + (w^{n} \cdot \nabla).
\]
We have
\begin{align}
	&R_{11}^n - R_{11}^{n-1} \notag \\ 
	= & 
	D_w^{n-1} u^n - \frac{u^n - u^{n-1} \circ X_1(w^{n-1})}{\Delta t} 
	- \biggl[ D_w^{n-2} u^{n-1} - \frac{u^{n-1} - u^{n-2} \circ X_1(w^{n-2})}{\Delta t} \biggr] \notag \\
	& + [(w^n - w^{n-1})\cdot \nabla] u^n - [(w^{n-1} - w^{n-2})\cdot \nabla] u^{n-1} \notag \\
	= & \frac{\Delta t}{2} (D_w^{n-1})^2 u^n
	- \Delta t^2 \int_0^1 \frac{(1-s)^2}{2} ( D_w^{n-1} )^3 u(y_1(\cdot,s), t_1(s)) ds \notag \\
	&- \frac{\Delta t}{2} ( D_w^{n-2})^2 u^{n-1}
	+ \Delta t^2 \int_0^1 \frac{(1-s)^2}{2} (D_w^{n-2})^3 u(y_2(\cdot,s), t_2(s)) ds \notag \\
	& + [(w^n - w^{n-1})\cdot \nabla] u^n - [(w^{n-1} - w^{n-2})\cdot \nabla] u^{n-1}. \label{eq:dtR11est1}
\end{align}
We denote the $j$-th term in \eqref{eq:dtR11est1} by $R^n_{11j}$.
We use \eqref{eq:dthalfest} to have the following bound:
\begin{align*}
	& \| R^n_{111} - R^n_{113} \|_0 
	\leq  \frac{\Delta t^{3/2}}{2} \biggl\| \frac{\partial}{\partial t} \biggl\{ \Bigl[ \frac{\partial}{\partial t} + (w(\cdot - \Delta t) \cdot \nabla) \Bigr]^2 u  \biggr\} \biggr\|_{L^2(t^{n-1},t^n; L^2)} \\
	\leq & c (\|w\|_{W^{1,\infty}(L^\infty)}) \Delta t^{3/2} \| u \|_{Z^3(t^{n-1}, t^n)}.
\end{align*}

For the second term,
\begin{align*}
	\|R_{112}^n\|_0
	& \leq \Delta t^2 \int_0^1 \biggl\| \frac{(1-s)^2}{2} \Bigl( \frac{\partial}{\partial t} + w^{n-1} \cdot \nabla \Bigr)^3 u(y_1(\cdot,s), t_1(s)) \biggr\|_0 ds \\
	& \leq c(\|w^{n-1}\|_{0,\infty}) \Delta t^{3/2} \| u \|_{Z^3(t^{n-1}, t^n)},
\end{align*}
where we have used the transformation of independent variables from $x$ to $y$ and $s$ to $t$, and the estimate $|\det(\partial x/\partial y)| \leq 2$ by virtue of Lemma \ref{lemm:bijectiveMix}-(\ref{item:jacobiest}).
By the same argument
\[
	\|R_{114}^n\|_0 \leq c(\|w^{n-2}\|_{0,\infty}) \Delta t^{3/2} \| u \|_{Z^3(t^{n-2}, t^{n-1})}.
\]

For the fifth and the sixth term,
\begin{align*}
	& R_{115}^n - R_{116}^n \\
	= & \int_{t^{n-1}}^{t^n} \frac{\partial }{\partial t} \biggl\{ \Bigl[ (w(t) - w(t-\Delta t)) \cdot \nabla \Bigr] u(t) \biggr\} dt \\
	=&  \int_{t^{n-1}}^{t^n} \Bigl[ (w_t(t) - w_t(t-\Delta t)) \cdot \nabla \Bigr] u(t) dt
	+  \int_{t^{n-1}}^{t^n} \Bigl[ (w(t) - w(t-\Delta t)) \cdot \nabla \Bigr] u_t(t) dt \\
	= & \int_{t^{n-1}}^{t^n} \Bigl[ \int_{t-\Delta t}^t w_{tt}(s) ds \cdot \nabla \Bigr] u(t) dt
	+  \int_{t^{n-1}}^{t^n} \Bigl[ \int_{t-\Delta t}^t w_t(s) ds \cdot \nabla \Bigr] u_t(t) dt,
\end{align*}
where $\Phi_t = \frac{\partial \Phi}{\partial t}$.
Thus,
\begin{align*}
	\| R_{115}^n - R_{116}^n \|_0
	\leq & c(\| w \|_{W^{2,\infty}(L^\infty)}) \Delta t^{3/2} \| u \|_{H^1(t^{n-1}, t^n; H^1)}.
\end{align*}
Gathering these estimates, from \eqref{eq:dtR11est1}, we have
the conclusion \eqref{eq:dtR11}.

We prove \eqref{eq:dtR12}.
First, we decompose the residual function as follows:
\begin{align*}
	&\frac{1}{\Delta t}  [ ( v^{n} - v^{n-1} \circ X_1^{n-1} ) -  ( v^{n-1} - v^{n-2} \circ X_1^{n-1} ) ] \\
	&- \frac{1}{\Delta t}  [v^{n-2} \circ X_1^{n-1} - v^{n-2} \circ X_1^{n-2} ]
	=: R^n_{41} - R^n_{42}.
\end{align*}
Using $y(x,t,\tau) := x - w^{n-1}(x) (t- \tau)$, we have
\begin{align*}
	R_{41}(x) 
	= & 
	\frac{1}{\Delta t} \int_{t^{n-1}}^{t^n} \Bigl[ v_t(x,t) - v_t(x-w^{n-1}(x)\Delta t, t - \Delta t) \Bigr] dt \\
	= & \frac{1}{\Delta t} \int_{t^{n-1}}^{t^n} \int_{t - \Delta t}^{t} D_w^{n-1} v_t (y(x,t,\tau),t) d\tau dt,
\end{align*}
and thus
\begin{equation}
	\| R_{41} \|_0 \leq c \Delta t^{1/2} ( \|v\|_{H^2(t^{n-2},t^n; L^2)} + \|w^{n-1}\|_{0,\infty} \| v \|_{H^1(t^{n-2}, t^n); H^1} ).
	\label{eq:r121complete}
\end{equation}
Here, we have again used the transformation of independent variables from $x$ to $y$ and the estimate $|\det(\partial x/\partial y)| \leq 2$.

The bound for $R^n_{42}$ is easily obtained by Lemma \ref{lemm:twofunc} with $q=2$, $p=\infty$, $p'=1$, $v = v^{n-2}$ and $w_i = w^{n-i}$, $i=1,2$:
\begin{equation}
\begin{split}
	\| R^n_{42} \|_0 &\leq c \| \nabla v^{n-2} \|_{0} \| w^{n-1} - w^{n-2} \|_{0,\infty}
	\leq c \Delta t \| \nabla v^{n-2} \|_{0} \|w\|_{W^{1,\infty}(L^\infty)}.
\end{split}
\label{eq:r122complete}
\end{equation}
The conclusion \eqref{eq:dtR12} follows from \eqref{eq:r121complete} and \eqref{eq:r122complete}.
\end{proof}


\paragraph{Acknowledgment}
The author was supported by Japan Society for the Promotion of Science under KAKENHI Grant Number JP18K13461.



\small

\begin{thebibliography}{10}

\bibitem{AchdouGuermond2000}
Y.~Achdou and J.-L. Guermond.
\newblock Convergence analysis of a finite element
  projection/{Lagrange--Galerkin} method for the incompressible
  {Navier--Stokes} equations.
\newblock {\em SIAM Journal on Numerical Analysis}, 37(3):799--826, 2000.

\bibitem{Badia2007}
S.~Badia and R.~Codina.
\newblock Convergence analysis of the {FEM} approximation of the first order
  projection method for incompressible flows with and without the inf-sup
  condition.
\newblock {\em Numerische Mathematik}, 107(4):533--557, 2007.

\bibitem{templates}
R.~Barrett, M.~Berry, T.~F. Chan, J.~Demmel, J.~Donato, J.~Dongarra,
  V.~Eijkhout, R.~Pozo, C.~Romine, and H.~{Van der Vorst}.
\newblock {\em Templates for the Solution of Linear Systems: Building Blocks
  for Iterative Methods, 2nd Edition}.
\newblock SIAM, Philadelphia, PA, 1994.

\bibitem{BenitezBermudez2011}
M.~Ben\'{i}tez and A.~Berm\'{u}dez.
\newblock A second order characteristics finite element scheme for natural
  convection problems.
\newblock {\em Journal of Computational and Applied Mathematics},
  235(11):3270--3284, 2011.

\bibitem{BrezziPitkaranta}
F.~Brezzi and J.~Pitk\"aranta.
\newblock On the stabilization of finite element approximations of the {Stokes}
  equations.
\newblock In W.~Hackbusch, editor, {\em Efficient Solutions of Elliptic
  Systems}, pages 11--19. Vieweg, 1984.

\bibitem{Burman2008}
E.~Burman.
\newblock Pressure projection stabilizations for {Galerkin} approximations of
  {Stokes'} and {Darcy's} problem.
\newblock {\em Numerical Methods for Partial Differential Equations},
  24(1):127--143, 2008.

\bibitem{Burman2017}
E.~Burman, A.~Ern, and M.~A. Fern\'andez.
\newblock Fractional-step methods and finite elements with symmetric
  stabilization for the transient {Oseen} problem.
\newblock {\em ESAIM: Mathematical Modelling and Numerical Analysis},
  51(2):487--507, 2017.

\bibitem{BurmanFernandez2008}
E.~Burman and M.~A. Fern\'andez.
\newblock {Galerkin} finite element methods with symmetric pressure
  stabilization for the transient {Stokes} equations: Stability and convergence
  analysis.
\newblock {\em SIAM Journal on Numerical Analysis}, 47(1):409--439, 2009.

\bibitem{ChenFeng2017}
G.~Chen and M.~Feng.
\newblock Analysis of solving {Galerkin} finite element methods with symmetric
  pressure stabilization for the unsteady {Navier}-{Stokes} equations using
  conforming equal order interpolation.
\newblock {\em Advances in Applied Mathematics and Mechanics}, 9(2):362--377,
  2017.

\bibitem{Ciarlet}
P.G. Ciarlet.
\newblock {\em The Finite Element Method for Elliptic Problems}, volume~40 of
  {\em Classics in Applied Mathematics}.
\newblock SIAM, 2002.

\bibitem{Clement}
Ph. Cl\'ement.
\newblock Approximation by finite element functions using local regularization.
\newblock {\em RAIRO Analyse num\'erique}, 9(R2):77--84, 1975.

\bibitem{deFrutos2016}
J.~de~Frutos, B.~Garc{\'i}a-Archilla, V.~John, and J.~Novo.
\newblock Grad-div stabilization for the evolutionary {Oseen} problem with
  inf-sup stable finite elements.
\newblock {\em Journal of Scientific Computing}, 66(3):991--1024, 2016.

\bibitem{DeFrutos2019IMA}
J.~de~Frutos, B.~Garc\'ia-Archilla, V.~John, and J.~Novo.
\newblock Error analysis of non inf-sup stable discretizations of the
  time-dependent {Navier}--{Stokes} equations with local projection
  stabilization.
\newblock {\em IMA Journal of Numerical Analysis}, 39(4):1747--1786, 2019.

\bibitem{deFrutos2018projNS}
J.~de~Frutos, B.~Garc{\'i}a-Archilla, and J.~Novo.
\newblock Error analysis of projection methods for non inf-sup stable mixed
  finite elements: The {Navier}--{Stokes} equations.
\newblock {\em Journal of Scientific Computing}, 74(1):426--455, 2018.

\bibitem{deFrutos2018projStokes}
J.~de~Frutos, B.~Garc{\'i}a-Archilla, and J.~Novo.
\newblock Error analysis of projection methods for non inf-sup stable mixed
  finite elements. {The} transient {Stokes} problem.
\newblock {\em Applied Mathematics and Computation}, 322:154--173, 2018.

\bibitem{DeFrutos2019JSC}
J.~de~Frutos, B.~Garc\'ia-Archilla, and J.~Novo.
\newblock Fully discrete approximations to the time-dependent
  {Navier}--{Stokes} equations with a projection method in time and grad-div
  stabilization.
\newblock {\em Journal of Scientific Computing}, 80(2):1330--1368, 2019.

\bibitem{DeFrutos2019JSCCorri}
J.~de~Frutos, B.~Garc\'ia-Archilla, and J.~Novo.
\newblock Corrigenda: Fully discrete approximations to the time-dependent
  {Navier}--{Stokes} equations with a projection method in time and grad-div
  stabilization.
\newblock {\em Journal of Scientific Computing}, 88(40), 2021.

\bibitem{DouglasRussell1982}
J.~Douglas, Jr. and T.~Russell.
\newblock Numerical methods for convection-dominated diffusion problems based
  on combining the method of characteristics with finite element or finite
  difference procedures.
\newblock {\em SIAM Journal on Numerical Analysis}, 19(5):871--885, 1982.

\bibitem{Franca1988}
L.~P. Franca and T.~J.~R. Hughes.
\newblock Two classes of mixed finite element methods.
\newblock {\em Computer Methods in Applied Mechanics and Engineering},
  69(1):89--129, 1988.

\bibitem{FrancaStenberg1991}
L.P. Franca and R.~Stenberg.
\newblock Error analysis of some {Galerkin} least squares methods for the
  elasticity equations.
\newblock {\em SIAM Journal on Numerical Analysis}, 28(6):1680--1697, 1991.

\bibitem{GJN2021kine}
B.~Garc{\'i}a-Archilla, V.~John, and J.~Novo.
\newblock On the convergence order of the finite element error in the kinetic
  energy for high {Reynolds} number incompressible flows.
\newblock {\em Computer Methods in Applied Mechanics and Engineering},
  385:114032, 2021.

\bibitem{Garcia-archilla2020}
B.~Garc\'ia-Archilla, V.~John, and J.~Novo.
\newblock Symmetric pressure stabilization for equal-order finite element
  approximations to the time-dependent {Navier}--{Stokes} equations.
\newblock {\em IMA Journal of Numerical Analysis}, 41(2):1093--1129, 2021.

\bibitem{Guermond1996}
J.-L. Guermond.
\newblock Some implementations of projection methods for {Navier}-{Stokes}
  equations.
\newblock {\em ESAIM: Mathematical Modelling and Numerical Analysis},
  30(5):637--667, 1996.

\bibitem{Guermond2006}
J.-L. Guermond, P.~Minev, and J.~Shen.
\newblock An overview of projection methods for incompressible flows.
\newblock {\em Computer Methods in Applied Mechanics and Engineering},
  195:6011--6045, 2006.

\bibitem{GuermondMinev2003}
J.-L. Guermond and P.~D. Minev.
\newblock Analysis of a projection/characteristic scheme for incompressible
  flow.
\newblock {\em Communications in Numerical Methods in Engineering},
  19(7):535--550, 2003.

\bibitem{Guermond1998}
J.-L. Guermond and L.~Quartapelle.
\newblock On the aproximation of the unsteady {Navier}--{Stokes} equations by
  finite element projection methods.
\newblock {\em Numerische Mathematik}, 80:207--238, 1998.

\bibitem{FreeFemCite}
F.~Hecht.
\newblock New development in {FreeFem++}.
\newblock {\em Journal of Numerical Mathematics}, 20(3-4):251--265, 2012.

\bibitem{HeywoodRannacher1990}
J.~G. Heywood and R.~Rannacher.
\newblock Finite-element approximation of the nonstationary {Navier}-{Stokes}
  problem. {Part IV}: Error analysis for second-order time discretization.
\newblock {\em SIAM Journal on Numerical Analysis}, 27(2):353--384, 1990.

\bibitem{JohnNovo2015}
V.~John and J.~Novo.
\newblock Analysis of the pressure stabilized {Petrov}--{Galerkin} method for
  the evolutionary {Stokes} equations avoiding time step restrictions.
\newblock {\em SIAM Journal on Numerical Analysis}, 53(2):1005--1031, 2015.

\bibitem{9thArt}
M.~E. Laursen and M.~Gellert.
\newblock Some criteria for numerically integrated matrices and quadrature
  formulas for triangles.
\newblock {\em International Journal for Numerical Methods in Engineering},
  12(1):67--76, 1978.

\bibitem{Linke2013}
A.~Linke and L.~G. Rebholz.
\newblock On a reduced sparsity stabilization of grad-div type for
  incompressible flow problems.
\newblock {\em Computer Methods in Applied Mechanics and Engineering},
  261-262:142--153, 2013.

\bibitem{LukacovaNonlin}
M.~Luk\'{a}\v{c}ov\'{a}-Medvid'ov\'{a}, H.~Mizerov\'a, H.~Notsu, and M.~Tabata.
\newblock Numerical analysis of the {Oseen}-type {Peterlin} viscoelastic model
  by the stabilized {Lagrange}-{Galerkin} method. {Part I}: A nonlinear scheme.
\newblock {\em ESAIM: Mathematical Modelling and Numerical Analysis},
  51(5):1637--1661, 2017.

\bibitem{LukacovaLin}
M.~Luk\'{a}\v{c}ov\'{a}-Medvid'ov\'{a}, H.~Mizerov\'a, H.~Notsu, and M.~Tabata.
\newblock Numerical analysis of the {Oseen}-type {Peterlin} viscoelastic model
  by the stabilized {Lagrange}-{Galerkin} method. {Part II}: A linear scheme.
\newblock {\em ESAIM: Mathematical Modelling and Numerical Analysis},
  51(5):1663--1689, 2017.

\bibitem{Misawa2016}
A.~Misawa.
\newblock {\em Error estimates of an {Euler} approximated
  characteristics/projection finite element scheme for the incompressible
  {Navier}--{Stokes} equations and its application}.
\newblock Master's thesis, Waseda University, Japan, 2016.
\newblock In Japanese.

\bibitem{NotsuTabataOseen2015}
H.~Notsu and M.~Tabata.
\newblock Error estimates of a pressure-stabilized characteristics finite
  element scheme for the {Oseen} equations.
\newblock {\em Journal of Scientific Computing}, 65(3):940--955, 2015.

\bibitem{OlshanskiiReusken}
M.A. Olshanskii and A.~Reusken.
\newblock Grad-div stablilization for {Stokes} equations.
\newblock {\em Mathematics of Computation}, 73:1699--1718, 2004.

\bibitem{Pironneau1982}
O.~Pironneau.
\newblock On the transport-diffusion algorithm and its applications to the
  {Navier-Stokes} equations.
\newblock {\em Numerische Mathematik}, 38:309--332, 1982.

\bibitem{RuiTabata2002}
H.~Rui and M.~Tabata.
\newblock A second order characteristic finite element scheme for
  convection-diffusion problems.
\newblock {\em Numerische Mathematik}, 92:161--177, 2002.

\bibitem{Suli}
E.~S\"{u}li.
\newblock Convergence and nonlinear stability of the {Lagrange-Galerkin} method
  for the {Navier-Stokes} equations.
\newblock {\em Numerische Mathematik}, 53(4):459--483, 1988.

\bibitem{Tabata2018}
M.~Tabata and S.~Uchiumi.
\newblock An exactly computable {Lagrange}--{Galerkin} scheme for the
  {Navier}--{Stokes} equations and its error estimates.
\newblock {\em Mathematics of Computation}, 87(309):39--67, 2018.

\bibitem{TSTeng}
K.~Tanaka, A.~Suzuki, and M.~Tabata.
\newblock A characteristic finite element method using the exact integration.
\newblock {\em Annual Report of Research Institute for Information Technology
  of Kyushu University}, 2:11--18, 2002.
\newblock (Japanese).

\bibitem{Uchiumi2019}
S.~Uchiumi.
\newblock A viscosity-independent error estimate of a pressure-stabilized
  {Lagrange}--{Galerkin} scheme for the {Oseen} problem.
\newblock {\em Journal of Scientific Computing}, 80(2):834--858, 2019.

\end{thebibliography}
\end{document}